\documentclass[a4paper,fleqn]{cas-sc}
\usepackage{float}
\usepackage{soul}
\usepackage{graphicx}
\usepackage{amsthm}  
\usepackage{amsmath}
\usepackage{mathrsfs}                
\usepackage{amsbsy}
\usepackage{amssymb}   
\usepackage{mathtools} 
\usepackage{color}		
\usepackage{epsfig}
\usepackage{mathrsfs}
\usepackage[normalem]{ulem}
\usepackage{xcolor}

\usepackage{sidecap}
\usepackage{caption}
\usepackage{cases}
\usepackage{enumerate}

\newtheorem{theorem}{Theorem}[section]
\newtheorem{lemma}{Lemma}[section]

\newtheorem{proposition}{Proposition}[section]

\newtheorem{remark}{Remark}[section]
\let \hat=\widehat
\numberwithin{equation}{section}

\makeatletter
\renewenvironment{proof}[1][\proofname]{\par
	\pushQED{\qed}%
	\normalfont \topsep6\p@\@plus6\p@\relax
	\trivlist
	\item[\hskip\labelsep
	\itshape
	#1\@addpunct{.}]\ignorespaces
}{%
	\popQED\endtrivlist\@endpefalse
}
\makeatother

\begin{document}
	\let\WriteBookmarks\relax
	\def\floatpagepagefraction{1}
	\def\textpagefraction{.001}
	\shorttitle{Stability for Micropolar equations}
	\shortauthors{J. Qi \& L. Yu}
	
	\title [mode = title]{Stability for the 2D Micropolar equations near Couette flow via Green's function method}
	\tnotemark[1]   
	\tnotetext[1]{This document is the results of the research project funded by  the National Natural Science Foundation of China (No. 12271357 and 11831011).}
	

	\author[1]{Jie Qi}
	\cormark[1]
	\ead{24310454@tongji.edu.cn}
	
	
	\affiliation[1]{organization={School of Mathematical Sciences,
			Key Laboratory of Intelligent Computing and Applications (Ministry of Education), 
			Tongji University},
		addressline={Siping Road 1239}, 
		city={Shanghai},
		postcode={200092}, 
		country={CHINA}}
	
	\author[2]{Lei Yu}[orcid=0000-0003-2817-4742]
	
	\ead{yu_lei@tongji.edu.cn}
	
	\affiliation[2]{organization={School of Mathematical Sciences,
			Key Laboratory of Intelligent Computing and Applications (Ministry of Education), 
			Tongji University},
		addressline={Siping Road 1239}, 
		city={Shanghai},
		postcode={200092}, 
		country={CHINA}}

	\cortext[cor1]{Corresponding author}
	
	
	\begin{abstract}
		In this paper, we use the Green's function method to study the stability of solutions to micropolar equations in the whole space. An essential difficulty in estimating the Green’s function lies in having to treat a coupled system of equations with variable coefficients, which is a situation completely different from those encountered in previous Green’s function  that only involved scalar equations. In this paper, we firstly propose a dominant factor extraction method. After successfully obtaining the estimates of the Green's function for the micropolar equations, we prove that if the initial data $(m_0, \omega_0)$     satisfy $\|(m_0, \omega_{0})\|_{L^1\cap L^{\infty}}\leq c_0\mu^{3/4}$ for some small constant $c_0$ independent of viscosity $\mu$, then the solution remains of order $O(\mu^{3/4})$ due to the effect of the Couette flow. Meanwhile, we derive the decay estimates of the solution in the $L^p$ norm. More importantly, the dominant factor extraction method proposed provides a general approach to stability analysis of more complex coupled fluid systems.
	\end{abstract}

	
	\maketitle

	\section{Introduction}
	In this paper, we consider the Cauchy problem of two-dimensional micropolar equations:
	\begin{equation}\label{1.1}
		\begin{cases}
			\partial_t U + U \cdot \nabla U + \nabla P - (\nu + \kappa)\Delta U = 2\kappa\nabla^\perp W, \\
			\partial_t W + U \cdot \nabla W - \mu\Delta W + 4\kappa W = 2\kappa\nabla \times U,\\
			\mathrm{div}\, U = 0, \\
			U(0,x,y)=U_0(x,y), \ \ W(0,x,y)=W_0(x,y),
		\end{cases} 
	\end{equation}
	where $\nabla^{\perp}=(\partial_y, -\partial_x)$, $(x,y)\in\mathbb{R}^2, t>0$. Here, the unknown functions $U=(U^1(x,y,t),U^2(x,y,t))$, $P=P(x,y,t)$, and $W=W(x,y,t)$ represent the velocity, pressure, and microrotation of the fluid, respectively. The parameters $\nu\geq0$, $\kappa>0$ and $\mu\geq0$ are the Newtonian viscosity, the microrotation viscosity and the angular viscosity, respectively.
	
	The micropolar fluid theory is an important extension of the classical Navier–Stokes equations. It introduces microrotation and inertial spin into the velocity field and can be used to describe complex fluids with microstructure, particle rotation, anisotropy, and coexisting viscous dissipation and rotational dissipation, such as colloidal suspensions, polymer solutions, magnetic fluids, biological fluids, and flows in porous media. Conducting stability studies on two-dimensional micropolar equations has both theoretical significance and practical value. 
	On the one hand, it reveals the influence of microrotation on flow stability; on the other hand, it provides theoretical support for microscale flows, biological fluids, and complex engineering flows. For the discussion of the global existence of two-dimensional micropolar equations, one may refer to \cite{7,8,9,10,11,12}. In the case of bounded or periodic domains, under certain parameter constraints, the stability threshold of the two-dimensional micropolar equations  can be achieved at 2/3; for details, see \cite{13,14}.
	
	Since the early experiments conducted by O. Reynolds \cite{O.Reynolds.1883}, the stability of laminar flow at high Reynolds numbers and the instability mechanisms leading to the transition from laminar flow to turbulent flow have always been popular research topics. There is extensive literature about the study of the instability mechanisms in laminar flow turbulence transition \cite{S.Orszag.1980,Schmid.2001,Yaglom.2012}. To quantify the transition mechanism, Trefethen et al.\cite{Trefethen.1993} posed the transition threshold, which can be traced back to Kelvin \cite{Kelvin.1887}. Bedrossian, Germain and Masmoudi \cite{Bedrossian.2019} formalized the threshold mathematically.
	Given a norm $||\cdot||_X$, if there exists a constant $\gamma=\gamma(X)$ such that 
	$$||u_0||_X<C \nu^{\gamma}\Longrightarrow stability,$$
	$$||u_0||_X\gg C \nu^{\gamma}\Longrightarrow instability,$$
	then we  call $\gamma$ the transition threshold in the applied literature.
	
	Regarding the research on the transition threshold problem of fluids, previous studies mainly considered the cases of bounded domains or periodic domains\cite{Duguet.2010,Lundbladh.1994, S.Orszag.1980,Reddy.1998,Yaglom.2012}, and rarely took into account the unbounded domains, especially the whole space \cite{1key,2key,16}.  For the Couette flow stability threshold problem, one can see \cite{Bedrossian.2020,Bedrossian.2017,Bedrossian.2016,Bedrossian.2018}. 
	
	In this paper, we aim to establish a stability threshold for the micropolar equations near Couette flow in the whole space under certain parameter constraints. The general parametric assumptions will be of interest in our subsequent work. The whole-space setting enables us to isolate the intrinsic bulk stability mechanism of micropolar fluids without the interference of boundary effects. It is easy to see that the Couette flow 
	\[
	U_s = (y, 0),\quad W_s = -\frac{1}{2},\quad P_s = c
	\]
	is a stationary solution of (\ref{1.1}). Now we introduce the perturbation $U = u + U_s$,\ $W= \omega + W_s$, $P = p + P_s$, then $(u, \omega, p)$ satisfy
	
	\begin{equation}\label{1.2}
		\begin{cases}
			\partial_t u + y\partial_x u - (\nu + \kappa)\Delta u + \binom{u^2}{0} - 2\kappa\nabla^{\perp} \omega + \nabla p + u \cdot \nabla u = 0, \\
			\partial_t \omega + y\partial_x \omega - \mu\Delta \omega + 4\kappa\omega - 2\kappa\nabla \times u + u \cdot \nabla \omega = 0, \\
			\nabla \cdot u = 0,\\
			u(0,x,y)=u_0(x,y),\ \ \omega(0,x,y)=\omega_0(x,y),
		\end{cases} 
	\end{equation}
	where $u=(u^1(x,y,t),u^2(x,y,t)).$\\
	Now we introduce the vorticity $m=\nabla \times u=\partial_x u^2-\partial_y u^1$ and rewrite the above system as
	\begin{equation}\label{1.3}
		\begin{cases}
			\partial_t m + y\partial_x m - (\nu + \kappa)\Delta m + u \cdot \nabla m = -2\kappa\Delta\omega, \\
			\partial_t \omega + y\partial_x \omega - \mu\Delta \omega + u \cdot \nabla \omega = 2\kappa m - 4\kappa\omega, \\
			u = \nabla^\perp \psi,\ \Delta \psi = m, \nabla \cdot u=0,\\
			m(0,x,y)=m_0(x,y),\ \ \omega(0,x,y)=\omega_0(x,y),
		\end{cases} 
	\end{equation}
	where $\psi$ is the stream function. Define the linear operator
	\[
	A(D) = 
	\begin{pmatrix}
		(\nu + \kappa)\Delta & -2\kappa\Delta \\
		2\kappa & \mu\Delta - 4\kappa
	\end{pmatrix}.
	\]
	
	This allows us to express system (\ref{1.3}) as:
	\begin{equation}\label{1.4}
		\partial_t
		\begin{pmatrix}
			m \\
			\omega
		\end{pmatrix}
		+y\partial_x
		\begin{pmatrix}
			m \\
			\omega
		\end{pmatrix}
		- A(D)
		\begin{pmatrix}
			m \\
			\omega
		\end{pmatrix}
		+ u \cdot \nabla
		\begin{pmatrix}
			m \\
			\omega
		\end{pmatrix}
		= 0,
	\end{equation}
	and
	\begin{equation}\label{001}
		m(0,x,y)=m_0(x,y),\ \ \omega(0,x,y)=\omega_0(x,y).
	\end{equation}
	The main theorem of this paper is as follows:
	\begin{theorem}\label{thm1.1}
		Suppose  that $m_0, \omega_0\in L^1(\mathbb{R}^2)\cap L^{\infty}(\mathbb{R}^2)$ and take $\nu = \kappa =\frac{1}{2} \mu$ which satisfies $0<\mu \ll 1$. There exists constants $\mu_0$  and $c_0, C> 0$ independent  of $\mu$  so that if 
		\begin{equation}
			\|(m_0, \omega_{0})\|_{L^1\cap L^{\infty}}\leq c_0\mu^{\frac{3}{4}}
		\end{equation}
		for some sufficiently small $c_0,0<\mu\leq \mu_0,$ 
		then the solution of the system (\ref{1.4})-(\ref{001}) is global in time and satisfies the following stability estimate 
		\begin{equation}
			\|(m(t,\cdot, \cdot), \omega(t,\cdot,\cdot))\|_{L^\infty L^p}\leq Cc_0(1+t)^{-2(1-\frac{1}{p})}\mu^{\frac{3}{4}},\ \ \ \ \ \forall p \geq \frac{4}{3}.
		\end{equation}
	\end{theorem}
	\begin{remark}
		The stability threshold result obtained in Lemma \ref{thm1.1} is not optimal. The optimal stability threshold will be our focus in subsequent work.
	\end{remark}
	The proof of the main theorem in this paper relies primarily on the Green's function of the linear part of system (\ref{1.4}) and the corresponding estimates. In the literature, only a few works have discussed the stability of solutions to fluid mechanics equations using the Green's function method. For instance, the stability of solutions to the two-dimensional Navier–Stokes equations is studied in \cite{2key} via this method, and the stability of solutions to the two-dimensional Boussinesq equations is investigated in \cite{16}. A notable feature of the Green's functions used in these two works is that the linear parts of the corresponding systems are either scalar equations or can be completely decoupled. For fluid equations involving Couette flow, the corresponding equation in the frequency domain is a first‑order partial differential equation with variable coefficients, and its solution in the scalar case can be expressed analytically. In contrast, the micropolar equations considered in the present paper present a fundamentally different difficulty: their linear part in the frequency domain is a system of equations with variable coefficients.
	
	As is well known, even for an ordinary differential equation, a variable-coefficient scalar equation only admits an explicit analytical solution. However, variable-coefficient partial differential systems generally do not allow such explicit expressions; their solutions can only be represented via series expansions. Once the solution is available only in series form, pointwise estimates of the Green’s function—which are crucial for proving stability via the Green’s function method—become unattainable. Furthermore, even if all eigenvalues of a variable-coefficient system have negative real parts, one cannot conclude that its solutions are bounded. Classical counterexamples illustrating this fact can be found in \cite{Vinograd1957, Josic2002}.
	
	Undoubtedly, when studying fluid mechanics equations using the Green’s function method—such as the compressible Navier–Stokes equations, the micropolar equations discussed in this paper, and the multi-physics MHD equations—the corresponding linear parts all take the form of coupled systems with variable coefficients. Given that an explicit analytical expression for the Green’s function cannot be obtained in either the space–time or frequency domain, how to derive the required estimates of the Green’s function becomes a generally challenging problem. An important objective of this paper is to explore a universal approach to address this need.
	
	One of the main methods employed in this paper is the dominant factor extraction method (hereinafter referred to as DFEM). Specifically, we first identify and extract the dominant factor, deriving its corresponding analytical expression. The remaining secondary factor then satisfies a more complex system of equations. By exploiting the structure of the equations, we prove that the secondary factor has a corrective effect on the dominant factor but does not affect the estimates of the latter, thereby yielding refined estimates of the Green’s function. It can be seen that DFEM is a universal method for coupled systems, which is fundamentally important for estimating the Green’s function in the context of variable-coefficient systems.
	
	The remainder of this paper is organized as follows. In Section \ref{s:pre}, we derive the equations satisfied by the Green’s function of the linearized system and present some technical lemmas. In Section \ref{s:dfem}, we introduce the dominant factor extraction method and analyze the Green’s function of the linearized system. In Section \ref{s:estimateG}, we provide the estimates of the Green’s function. In Section \ref{s:NonlinearG}, we establish the nonlinear stability and the stability  of the solution.
	
	Throughout this paper, C denotes a generic constant which only depends on the system that may change from line to line.
	\section{Preliminaries}\label{s:pre}
	In this section, we first derive the equations satisfied by the Green's function, and then list some basic technical lemmas which will be needed in the subsequent parts of the paper.\par
	Taking $\nu = \kappa =\frac{1}{2} \mu$ which satisfies $0<\mu \ll 1$, and performing the time scaling transformation $t \to \dfrac{t}{\mu}$ and  defing $A=\dfrac1\mu$, then the Green's function of the linearized system of (\ref{1.4}) is the solution for the following system:
	\begin{equation}\label{3.00002}
		\left\{
		\begin{aligned}
			&\partial_t \mathbb{G} + A y \partial_x \mathbb{G} - \overline{B}(D)\mathbb{G}=0,\\
			&\mathbb{G}(x,y,0;x',y')=\delta(x-x',y-y'),
		\end{aligned}
		\right.
	\end{equation}
	where
	\[
	\overline{B}(\xi,\eta) = \begin{pmatrix} -(\xi^2+\eta^2) & (\xi^2+\eta^2) \\ 1 & -(\xi^2+\eta^2)-2 \end{pmatrix}.
	\]
	By taking the Fourier transform, let \( \hat{\mathbb{G}} = \hat{\mathbb{G}}_1 \hat{\mathbb{G}}_2 \), where \( \hat{\mathbb{G}}_1 \) and \( \hat{\mathbb{G}}_2 \) satisfy
	\begin{equation}\label{111}
		\begin{cases}
			\partial_t \hat{\mathbb{G}}_1 - A \xi \partial_\eta \hat{\mathbb{G}}_1 = 0, \\
			\hat{\mathbb{G}}_1\big|_{t=0} = \exp\bigl(-i(x'\xi + y'\eta)\bigr) I_{2\times2},
		\end{cases}
	\end{equation}
	and
	\begin{equation}\label{3.2}
		\begin{cases}
			\partial_t \hat{\mathbb{G}}_2 - A \xi \partial_\eta \hat{\mathbb{G}}_2 - \overline{B}(\xi,\eta) \hat{\mathbb{G}}_2 = 0, \\
			\hat{\mathbb{G}}_2\big|_{t=0} = I_{2\times2},
		\end{cases}
	\end{equation}
	respectively.
	From (\ref{111}), we can easily check that the solution of (\ref{111}) is as follows
	\begin{equation}\label{555}
		\hat{\mathbb{G}}_1=\exp\bigl(-i(x'\xi + y'(\eta+At\xi))\bigr) I_{2\times2}.
	\end{equation}
	Then taking coordinate transformation $ \tilde{\eta} = \eta + A t \xi
	$ to (\ref{3.2}), we obtain the new system satisfies
	\begin{equation}\label{3.3}
		\begin{cases}
			\partial_t \widehat{G} - B(\xi,\tilde{\eta},t)\widehat{G}=0,\\
			\hat{G}\big|_{t=0} = I_{2\times 2},
		\end{cases}
	\end{equation}
	where
	\[
	B(\xi,\tilde{\eta})=\begin{pmatrix} -(\xi^2+(\tilde{\eta}-A t\xi)^2) & (\xi^2+(\tilde{\eta}-A t\xi)^2) \\ 1 & -(\xi^2+(\tilde{\eta}-A t\xi)^2)-2 \end{pmatrix}.
	\]
	From (\ref{3.2}) and (\ref{3.3}), we know that 
	$$\hat{\mathbb{G}}_2(\xi, \eta, t)=\widehat{G}(\xi, \tilde{\eta}, t).$$
	The system (\ref{3.3}) is a variable-coefficient ordinary differential equations with respect to time $t$ which has parameters $(\xi, \tilde{\eta})$. We will discuss its explicit solution in Section 3, and present the corresponding estimates in Section 4.\par
	The following lemma is about the Young's inequality for the kernel form (see \cite{Bollob.1993}).
	\begin{lemma}\label{lem2.1}(Young's inequality )  Let $1 \leq p, q, r \leq \infty$ and $1+\frac{1}{r}=\frac{1}{q}+\frac{1}{p}.$ If the kernel $K(z,z')$ is a measurable function on $\mathbb{R}^d\times\mathbb{R}^d$ and satisfies
		$$
		\left\|K(\cdot,z')\right\|_{L^q}\leq A,\quad	\left\|K(z,\cdot)\right\|_{L^q}\leq B,
		$$
		we define  integral operator
		$$
		Tf(z)=\int_{\mathbb{R}^d}K(z,z')f(z')dz',
		$$
		then it holds that
		$$
		\left\|	Tf(z)\right\|_{L^r}\leq C	\left\|	f\right\|_{L^p},
		$$
		for	$f(z)\in L^{p}(\mathbb{R}^d),$	where $C=\max\left\lbrace A,B\right\rbrace$.
		Moreover, we have a much finer estimate:
		$$
		\left\|	Tf(z)\right\|_{L^r}\leq A^{\frac{q}{r}}	B^{q-\frac{q}{p}}\left\|f\right\|_{L^p}.
		$$
	\end{lemma}
	The following is the Hardy-Littlewood-Sobolev inequality.
	\begin{lemma}\label{lem2.3}
		Let $0<s<2,1<p<q<\infty,\dfrac{1}{q}+\dfrac{s}{2}=\dfrac{1}{p},$ then
		$$||\Lambda^{-s}f||_{L^q}\leq C||f||_{L^p}.$$
	\end{lemma}
	\begin{proof}
		See \cite{Stein.1970}  [ p.119, Theorem 1].
	\end{proof}	
	\begin{lemma}\label{lem:3.1} For any $\alpha>
		0$, there exists a constant $C_\alpha>0$ such that 
		\begin{equation*}
			\begin{array}{rl}
				\int^t_0|\eta+As\xi|^\alpha ds\geq C_\alpha(|\eta|^{\alpha}+(At)^{\alpha}|\xi|^{\alpha})t.
			\end{array}
		\end{equation*}
	\end{lemma}
	The following is the Gronwall's inequality.
	\begin{lemma}\cite{amann1990ode}  \label{lem:001}(Gronwall's inequality)
		Let $J$ be an interval in $R$, $t_0 \in J$, $a(t) = a_0(|t-t_0|)$, where $a_0 \in C(\mathbb{R}_+,\mathbb{R}_+)$ is a monotone increasing function, and assume that
		\[
		u(t) \leq a(t) + \left|\int_{t_0}^{t} \beta(s)u(s)\,ds\right|, \quad \forall t \in J.
		\]
		Then we obtain the estimate
		\[
		u(t) \leq a(t) e^{\left|\int_{t_0}^{t} \beta(s)\,ds\right|}, \quad \forall t \in J.
		\]
	\end{lemma}

	Now we give the following abstract bootstrap argument.
	\begin{lemma}\label{lem2.2}\cite{8_pa}
		Let $T>0. $ Assume that two statements C(t) and H(t) with $t\in [0,T]$ satisfy the following conditions\\
		(a) If H(t) holds for some $t\in [0,T]$, then C(t) holds for the same t;\\
		(b) If C(t) holds for some $t_{0}\in [0,T]$, then H(t) holds for t in a neighborhood of $t_{0};$\\
		(c) If C(t) holds for $t_{m}\in [0,T]$ and $t_{m}\rightarrow t,$ then C(t) holds;\\
		(d) C(t) holds for at least one $t_{1}\in [0,T].$\\
		Then C(t) holds on $[0,T].$
	\end{lemma} 
	\section{Dominant Factor Extraction Method}\label{s:dfem}
	For linear ordinary differential equations with variable coefficients, the investigation of their solutions poses far greater challenges than those with constant coefficients. On the one hand, it is generally impossible to construct explicit solutions via algebraic approaches such as eigenvalue decomposition and matrix exponentiation; instead, the proof of the solutions existence solely relies on functional analysis and iterative theories, leading to substantial analytical difficulties. On the other hand, the instantaneous eigenvalues of the coefficient matrix only reflect local properties, which are insufficient to characterize the global long-term evolutionary behavior of solutions. Even if the real parts of all instantaneous eigenvalues are constantly negative, solutions may still exhibit unbounded exponential growth \cite{Vinograd1957,Josic2002}. Furthermore, the lack of universal methods for constructing Lyapunov functions and straightforward integral estimation techniques makes the simultaneous proof of the solutions existence and global boundedness a core challenge in the qualitative analysis of differential systems with variable coefficients.\par
	In the following, we consider a two-dimensional linear system with time-varying coefficients
	\[
	\dot{\mathbf{x}}(t) = A(t)\mathbf{x}(t),
	\]
	where the coefficient matrix is given by
	\[
	A(t)=
	\begin{pmatrix}
		-1 + \dfrac{3}{2}\cos^2t & 1+\dfrac{3}{2}\cos t\sin t \\[4pt]
		-1+\dfrac{3}{2}\cos t\sin t & -1+\dfrac{3}{2}\sin^2t
	\end{pmatrix}.
	\]
	It can be verified that the instantaneous eigenvalues of $A(t)$ are always
	\[
	\lambda_1 = \lambda_2 = -1,
	\]
	whose real parts are strictly negative for all $t\in\mathbb{R}$.
	However, the system possesses an unbounded solution
	\[
	\mathbf{x}(t)=e^{\frac{t}{2}}
	\begin{pmatrix}
		\cos t \\
		-\sin t
	\end{pmatrix},
	\]
	which grows exponentially as $t\to+\infty$.
	This classic counterexample \cite{Vinograd1957} shows that for linear systems with variable coefficients,
	the condition that all instantaneous eigenvalues have negative real parts
	does not guarantee the stability of the zero solution.
	
	\bibliographystyle{plain}
	
	Returning to the system (\ref{3.3}), as mentioned earlier, it is difficult to obtain its explicit solution, and it is impossible to derive the solution expression solely from its characteristic roots. This paper attempts to exploit the structure of the equation to separate its dominant factors, and to prove that the behavior of the solution is determined by these dominant factors. The method we used is called the ``\textbf{dominant factor extraction method}'' (\textbf{DFEM}).\par
	Define
	\[
	B(\xi,\tilde{\eta})=B_I+B_J=\begin{pmatrix} -1 & 1 \\ 0 & -1 \end{pmatrix}(\xi^2+(\tilde{\eta}-A t\xi)^2) + \begin{pmatrix} 0 & 0 \\ 1 & -2 \end{pmatrix}.
	\]
	We introduce $\hat{G} \eqqcolon \hat{G}_I \hat{G}_J,$
	and $\hat{G}_I $ solves the following equations:
	\begin{equation}\label{3.4}
		\begin{cases}
			\partial_t \hat{G}_I - B_I \hat{G}_I = 0,\\
			\hat{G}_I\big|_{t=0} = I_{2\times 2},
		\end{cases}
	\end{equation}
	where
	\[
	B_I = \begin{pmatrix} -1 & 1\\ 0 & -1 \end{pmatrix}
	\bigl(\xi^2 + (\tilde{\eta} - At\xi)^2\bigr).
	\]
	We refer to $\hat{G}_I$ as the \textbf{dominant factor} of the Green's function. For $\hat{G}_I$, we have the following lemma.
	\begin{lemma}\label{lem3.1}
		The solution of the equation (\ref{3.4}) is
		\[
		\hat{G}_I
		=
		\exp\!\left(-\int_0^t \bigl(\xi^2 + (\tilde{\eta} - As\xi)^2\bigr)ds\right)D,
		\]
		where
		\[
		D=
		\begin{pmatrix}
			1 & \displaystyle\int_0^t \bigl(\xi^2 + (\tilde{\eta} - As\xi)^2\bigr)ds\\[6pt]
			0 & 1
		\end{pmatrix}.
		\]
	\end{lemma}
	\begin{proof}
		We define
		\begin{equation}
			\hat{G}_I = \begin{pmatrix}
				\hat{G}_{I1} & \hat{G}_{I3} \\
				\hat{G}_{I2} & \hat{G}_{I4}
			\end{pmatrix}
		\end{equation}
		We first consider $	\hat{G}_{I3}$ and $	\hat{G}_{I4}$,
		then $\hat{G}_{I4}$ satisfies
		\[
		\begin{cases}
			\partial_t \hat{G}_{I4}+\bigl(\xi^2 + (\tilde{\eta} - At\xi)^2\bigr)\hat{G}_{I4}=0,\\
			\hat{G}_{I4}\big|_{t=0}=1,
		\end{cases}
		\]
		and hence
		\[
		\hat{G}_{I4}
		=
		\exp\!\left(-\int_0^t \bigl(\xi^2 + (\tilde{\eta} - As\xi)^2\bigr)ds\right).
		\]
		Then $\hat{G}_{I3}$ satisfies
		\[
		\begin{cases}
			\partial_t \hat{G}_{I3}+\bigl(\xi^2 + (\tilde{\eta} - At\xi)^2\bigr)\hat{G}_{I3}
			-\bigl(\xi^2 + (\tilde{\eta} - At\xi)^2\bigr)\hat{G}_{I4}=0,\\
			\hat{G}_{I3}\big|_{t=0}=0.
		\end{cases}
		\]
		Multiplying by the integrating factor
		$\exp\bigl(\int_0^t (\xi^2 + (\tilde{\eta}- As\xi)^2)ds\bigr)$, we get
		\[
		\partial_t\left(
		\hat{G}_{I3}\exp\!\left(\int_0^t \bigl(\xi^2 + (\tilde{\eta}- As\xi)^2\bigr)ds\right)
		\right)
		=
		\bigl(\xi^2 + (\tilde{\eta}- At\xi)^2\bigr).
		\]
		Integrating from $0$ to $t$ yields
		\[
		\hat{G}_{I3}
		=
		\exp\!\left(-\int_0^t \bigl(\xi^2 + (\tilde{\eta}- As\xi)^2\bigr)ds\right)
		\int_0^t \bigl(\xi^2 + (\tilde{\eta}- As\xi)^2\bigr)ds.
		\]
		By a similar argument, we have 
		\[
		\hat{G}_{I1}
		=
		\exp\!\left(-\int_0^t \bigl(\xi^2 + (\tilde{\eta}- As\xi)^2\bigr)ds\right).
		\]
		and
		\[
		\hat{G}_{I2}=0.
		\]
		Finally, we obtain
		\[
		\hat{G}_I
		=
		\exp\!\left(-\int_0^t \bigl(\xi^2 + (\tilde{\eta}- As\xi)^2\bigr)ds\right)D,
		\]
		where
		\[
		D=
		\begin{pmatrix}
			1 & \displaystyle\int_0^t \bigl(\xi^2 + (\tilde{\eta}- As\xi)^2\bigr)ds\\[6pt]
			0 & 1
		\end{pmatrix}.
		\]
	\end{proof}
	We note that $\hat{G}_2= \widehat{G}_I \widehat{G}_J$, substituting this equality into (\ref{3.3}), then we have
	\[
	(\partial_t \widehat{G}_I)\widehat{G}_J + \widehat{G}_I (\partial_t \widehat{G}_J) - (B_I+B_J)\widehat{G}_I \widehat{G}_J = 0.
	\]
	From (\ref{3.4}), we get
	\[
	\widehat{G}_I (\partial_t \widehat{G}_J) - B_J \widehat{G}_I \widehat{G}_J = 0,
	\]
	thus
	\[
	\partial_t \widehat{G}_J - (\widehat{G}_I^{-1} B_J \widehat{G}_I)\widehat{G}_J = 0.
	\]
	Note that $\widehat{G}_I^{-1} B_J \widehat{G}_I = D^{-1}B_J D$. Then we can compute that
	\[
	D^{-1}B_J D = \begin{pmatrix} -P(t) & -P(t)(P(t)-2) \\ 1 & P(t)-2 \end{pmatrix}
	\]
	where
	\[
	P(t) = \int_0^t (\xi^2+(\tilde{\eta}-A s\xi)^2) ds
	\]
	Then we consider the following ODE:
	\begin{equation}\label{3.5}
		\begin{cases}
			\partial_t \hat{G}_J + \begin{pmatrix} P(t) & P(t)(P(t)-2) \\ -1 & -(P(t)-2) \end{pmatrix} \hat{G}_J = 0,  \\
			\hat{G}_J\big|_{t=0} = I_{2\times 2},
		\end{cases}
	\end{equation}
	where $P(t) = \int_0^t \left(\xi^2 + (\tilde{\eta}- A s \xi)^2\right) ds$.
	Let $\hat{G}_J = \begin{pmatrix} g_{11} & g_{12} \\ g_{21} & g_{22} \end{pmatrix}$. 
	For $\hat{G}_J$, we have the following estimate.
	\begin{lemma}\label{lem3.2}
		For the solution of (\ref{3.5}), we have the following estimate
		$$|g_{11}|, |g_{12}|\leq C (e^{\frac{1}{2}P(t)}P(t)+1),\ \ |g_{21}|, |g_{22}|\leq e^{\frac{1}{2}P(t)}.$$
	\end{lemma}
	\begin{proof}
		We first consider $g_{11}$ and $g_{21}$, which satisfy the following equations:
		\begin{equation}\label{3.03}
			\begin{cases}
				\partial_t g_{11}+ P(t) g_{11} + P(t)(P(t)-2) g_{21} = 0, \\
				\partial_t g_{21} - g_{11} - (P(t)-2) g_{21} = 0,\\
				g_{11}\big|_{t=0} = 1,\quad g_{21}\big|_{t=0} = 0.
			\end{cases}
		\end{equation}
		From (\ref{3.03}), we have
		\[
		\begin{aligned}
			\partial_t (g_{11} + P(t) g_{21}) &= \partial_t g_{11} + P(t) \partial_t g_{21} + P'(t) g_{21} \\
			&= P'(t) g_{21}.
		\end{aligned}
		\]
		
		Let $w = g_{11} + P(t) g_{21}$. Then (\ref{3.03}) becomes
		\begin{equation}\label{3.7}
			\begin{cases}
				\partial_t w - P'(t) g_{21} = 0, \\
				\partial_t g_{21} - w + 2 g_{21} = 0, \\
				w\big|_{t=0} = 1, \quad g_{21}\big|_{t=0} = 0.
			\end{cases}
		\end{equation}
		
		From the second equation of (\ref{3.7}), we get $\partial_t^2 g_{21} - \partial_t w + 2 \partial_t g_{21} = 0$.
		Substituting into the first equation of (\ref{3.7}), we obtain
		\[
		\partial_t^2 g_{21} - P'(t) g_{21} + 2 \partial_t g_{21} = 0, \quad g_{21}\big|_{t=0} = 0, \quad \partial_t g_{21}\big|_{t=0} = 1.
		\]
		
		Let $g_{21} = e^{-t} F$, then we have
		\[
		\begin{aligned}
			\partial_t g_{21} &= -e^{-t} F + e^{-t} F', \\
			\partial_t^2 g_{21} &= e^{-t} F - 2 e^{-t} F' + e^{-t} F''.
		\end{aligned}
		\]
		
		This gives $e^{-t} \left(F'' - F - P'(t) F\right) = 0$, i.e.,
		\[
		F'' - (1 + P'(t)) F = 0.
		\]
		
		From the initial conditions of $g_{21}$ and its derivative, we have for $F$:
		\begin{equation}\label{3.8}
			\begin{cases}
				F'' - (1 + P'(t)) F = 0,  \\
				F(0) = 0, \quad F'(0) = 1.
			\end{cases}
		\end{equation}
		
		We now verify that
		\[
		F(t) = \frac{e^t - e^{-t}}{2} + \int_0^t \frac{e^{t-s} - e^{-(t-s)}}{2} P'(s) F(s) ds
		\]
		satisfies the equation (\ref{3.8}).
		
		First, it is clear that $F(0) = 0$.
		Next,
		\[
		F'(t) = \frac{e^t + e^{-t}}{2} + \int_0^t \frac{e^{t-s} + e^{-(t-s)}}{2} P'(s) F(s) ds,
		\]
		so $F'(0) = 1$.
		
		For $F''(t)$, we have
		\[
		\begin{aligned}
			F''(t) &= \frac{e^t - e^{-t}}{2} + P'(t) F(t) + \int_0^t \frac{e^{t-s} - e^{-(t-s)}}{2} P'(s) F(s) ds \\
			&= F(t) + P'(t) F(t),
		\end{aligned}
		\]
		which means $F$ satisfies $F'' - (1 + P'(t)) F = 0$.
		Thus,
		\[
		g_{21} =\frac{1 - e^{-2t}}{2} + \int_0^t \frac{1 - e^{-2(t-s)}}{2} P'(s) F(s) ds.
		\]
		Simple calculation gives
		\[
		|g_{21}(t)| \leq \frac{1}{2} + \frac{1}{2} \int_0^t P'(s) |g_{21}(s)| ds,
		\]
		and by the Gronwall's inequality in Lemma \ref{lem:001}, we have
		\[
		\begin{aligned}
			|g_{21}(t)| &\leq \frac{1}{2} \exp\left( \frac{1}{2} \int_0^t P'(s) ds \right) \\
			&= \frac{1}{2} \exp\left( \frac{1}{2} P(t) \right).
		\end{aligned}
		\]
		Next, we onsider the estimate for \( g_{11} \). Note that
		\[
		\begin{aligned}
			\partial_t\bigl(g_{11} + P(t)g_{21}\bigr)
			&= \partial_t g_{11} + P(t)\,\partial_t g_{21} + P'(t)g_{21} \\
			&= P'(t)g_{21}.
		\end{aligned}
		\]
		Integrating both sides with respect to \( t \), we obtain
		\[
		g_{11}(t) + P(t)g_{21}(t) = 1 + \int_0^t P'(s)g_{21}(s)\,ds.
		\]
		Hence,
		\[
		\begin{aligned}
			|g_{11}(t)|
			&\leq P(t)|g_{21}(t)| + C\,e^{\frac{1}{2}P(t)}\int_0^t P'(s)\,ds + 1 \\
			&\leq C\bigl(P(t)e^{\frac{1}{2}P(t)} + 1\bigr).
		\end{aligned}
		\]
		Similarly, by calculation we obtain
		\[
		g_{22} = e^{-2t} + \int_{0}^{t} \frac{1 - e^{-2(t-s)}}{2} P'(s) g_{22}(s) \, ds.
		\]
		By the Lemma \ref{lem:001}, we have
		\[
		|g_{22}| \leq e^{\frac{1}{2}P(t)}.
		\]
		Substituting into the equation satisfied by \(g_{12}\) and using the integrating factor method, we obtain
		\[
		|g_{12}| \leq C e^{\frac{1}{2}P(t)}  P(t).
		\]
	\end{proof}
	From Lemma \ref{lem3.1} and Lemma \ref{lem3.2}, we see that the form of the solution to equation (\ref{3.2}) is mainly determined by its dominant factor. This conclusion can be summarized in the following proposition.
	
	\begin{proposition}\label{prop:3.3}
		The solution $\hat{\mathbb{G}}_2$ to equation (\ref{3.2}) has the following estimate:
		\begin{equation*}
			|\widehat{\mathbb{G}}_2(\xi, \eta,t)|
			\leq \exp{(-C((|\xi|^2t(1+(At)^2)+|\eta|^2t))}I_{2\times 2}.     
		\end{equation*}
	\end{proposition}
	\begin{proof}
		From the coordinate transformation $ \tilde{\eta} = \eta + A t \xi
		$, we have
		\[
		P(t) = \int_0^t (\xi^2+(\tilde{\eta}-A s\xi)^2) ds=\int_0^t (\xi^2+(\eta+A(t-s)\xi)^2) ds= \int_0^t (\xi^2+(\eta+As\xi)^2) ds,
		\]
		then from $\hat{G} = \hat{G}_I \hat{G}_J$, the expression for $\hat{G}_I$, the estimate of $\hat{G}_J$, Lemma \ref{lem:3.1}, we have
		\begin{equation*}
			|\widehat{\mathbb{G}}_2(\xi, \eta, t)|
			\leq \exp{(-C((|\xi|^2t(1+(At)^2)+|\eta|^2t))}I_{2\times 2},
		\end{equation*}
		then we completed the proof of the lemma.
	\end{proof}
	
	\section{The Estimate of the Green's Function}\label{s:estimateG}
	With the expression for $\hat{G}_I$ and the estimate for $\hat{G}_J$ at hand , we will use Young's inequality to obtain the $L^p$ ($p\geq2$) norm estimate of the Green's function $\mathbb{G}$.
	
	\begin{lemma}\label{lem:3.2} For any non-negative integers $k=k_1+k_2$, $1\leq q \leq \infty,$ we have the following estimate
		\begin{equation}\label{eq:2.7}
			\||\xi|^{k_1}|\eta|^{k_2}\widehat{\mathbb{G}}_2(\cdot, \cdot, t)\|_{L^q} \leq
			Ct^{-\frac{1}{q}-\frac{k}{2}}(1+(At)^2)^{-\frac{1}{2}(\frac{1}{q}+k_1)}.
		\end{equation}
	\end{lemma}
	\begin{proof}
		For $k=k_1+k_2$, from Proposition \ref{prop:3.3}, we can obtain
		\begin{equation*}
			\begin{array}{rl}
				&(\int_{\mathbb{R}^2}|\xi|^{qk_1}|\eta|^{qk_2}\exp{(-Cq(|\xi|^2t(1+(At)^2)+|\eta|^2t))}d\xi d\eta)^{\frac{1}{q}} \\
				\leq &Ct^{-\frac{1}{q}-\frac{k}{2}}(1+(At)^2)^{-\frac{1}{2}(\frac{1}{q}+k_1)}t.
			\end{array}
		\end{equation*}
		Thus, we proved the lemma.
	\end{proof}
	\begin{lemma}\label{lem:3.3} For $p\geq 2, k=k_1+k_2$, we have the following estimates
		\begin{equation}\label{eq:2.10}
			\|\partial^{k_1}_{x}\partial^{k_2}_{y}\mathbb{G}(\cdot, \cdot, t; x^\prime,y^\prime)\|_{L^p}\leq Ct^{-(1-\frac{1}{p})-\frac{k}{2}}(1+(At)^2)^{-\frac{k_1}{2}-\frac{1}{2}(1-\frac{1}{p})},
		\end{equation}
		\begin{equation}\label{eq:2.11}
			\|\partial^{k_1}_{x}\partial^{k_2}_{y}\mathbb{G}(x, y,t;\cdot, \cdot)\|_{L^p}\leq Ct^{-(1-\frac{1}{p})-\frac{k}{2}}(1+(At)^2)^{-\frac{k_1}{2}-\frac{1}{2}(1-\frac{1}{p})}.
		\end{equation}
	\end{lemma}
	\begin{proof}
		Parseval equality and Young's inequality gives
		\begin{equation}\label{eq:2.12}
			\begin{array}{rl}
				\|\partial^{k_1}_{x}\partial^{k_2}_{y}\mathbb{G}(\cdot, \cdot, t; x^\prime,y^\prime)\|_{L^2}&=\||\xi|^{k_1}|\eta|^{k_2}\widehat{\mathbb{G}}(\cdot, \cdot,t; x^\prime,y^\prime)\|_{L^2},\\
				|\partial^{k_1}_{x}\partial^{k_2}_{y}\mathbb{G}(\cdot, \cdot,t; x^\prime,y^\prime)|\leq &C\||\xi|^{k_1}|\eta|^{k_2}\widehat{\mathbb{G}}(\cdot, \cdot,t; x^\prime,y^\prime)\|_{L^1}.
			\end{array}
		\end{equation}
		Using the interpolation theorem, from \eqref{eq:2.12}, the defination of  $\hat{\mathbb{G}}_1$ in (\ref{555}) and Lemma \ref{lem:3.2}, we can obtain
		\begin{align}\label{eq:2.13}
			&\|\partial^{k_1}_{x}\partial^{k_2}_{y}\mathbb{G}(\cdot, \cdot,t; x^\prime,y^\prime)\|_{L^p}\notag\\
			\leq{}&C\|\partial^{k_1}_{x}\partial^{k_2}_{y}\mathbb{G}(\cdot, \cdot,t; x^\prime,y^\prime)\|^{2/p}_{L^2}\|\partial^{k_1}_{x}\partial^{k_2}_{y}\mathbb{G}(\cdot, \cdot, t; x^\prime,y^\prime)\|^{1-2/p}_{L^\infty}\notag\\
			\leq{}&C\||\xi|^{k_1}|\eta|^{k_2}\widehat{\mathbb{G}_2}(\cdot, \cdot, t)\|^{2/p}_{L^2}\||\xi|^{k_1}|\eta|^{k_2}\widehat{\mathbb{G}_2}(\cdot, \cdot, t)\|^{1-2/p}_{L^1}\notag\\
			\leq{}&Ct^{-(1-\frac{1}{p})-\frac{k}{2}}(1+(At)^2)^{-\frac{k_1}{2}-\frac{1}{2}(1-\frac{1}{p})}.
		\end{align}
		Thus \eqref{eq:2.10} is proved. To prove \eqref{eq:2.11}, we study the relationship between the two sets of the variable $(x, y) $ and $(x^\prime, y^\prime)$ of $\mathbb{G}(x, y, t; x^\prime, y^\prime) $, Taking the inverse Fourier transform on $\widehat{\mathbb{G}}(\xi, \eta,t; x^\prime,y^\prime)$,
		we have
		\begin{equation}\label{eq:2.14}
			\begin{array}{rl}
				&\mathbb{G}(x, y,t; x^\prime,y^\prime)=\int_{\mathbb{R}^2}e^{i(x\xi+y\eta)}\widehat{\mathbb{G}}(\xi, \eta,t; x^\prime,y^\prime)d\xi d\eta \\
				=&\int_{\mathbb{R}^2}e^{i((x-x^\prime-Aty^\prime)\xi+(y-y^\prime)\eta)}\mathbb{G}_2(\xi, \eta, t)d \xi d\eta\\\
				\triangleq &F(x-x^\prime-Aty^\prime, y-y^\prime,t).
			\end{array}
		\end{equation}
		Note the structure of the above formula $F(x-x^\prime-Aty^\prime, y-y^\prime,t)$, we can see that the proof of \eqref{eq:2.11} is exactly the same as the proof of \eqref{eq:2.10}. So the proof of the lemma is complete.
	\end{proof}
	\begin{remark}
		From equation \eqref{eq:2.14}, we know that the Green's function is a kernel function, which has not symmetry for variables $(x,y)$ and $(x', y')$, and we have 
		\begin{equation}\label{3.17}
			\begin{cases}
				\partial_x \mathbb{G}(x - x', y, t; y') = -\partial_{x'} \mathbb{G}(x - x', y, t; y'), \\
				\partial_y \mathbb{G}(x - x', y, t; y') = -\partial_{y'} \mathbb{G}(x - x', y, t; y') + A t \, \partial_{x'} \mathbb{G}(x - x', y, t; y').
			\end{cases}
		\end{equation}
	\end{remark}
	In what follows, we obtain estimates for $\|\mathbb{G}\|_{L^1}$ and $\|\nabla \mathbb{G}\|_{L^1}$ by pointwise estimates.
	\begin{lemma}\label{lem:3.5}
		For \( \mathbb{G}_2 \), we have the following pointwise estimate:
		\begin{equation}\label{777}
			|\mathbb{G}_2(x,y,t)|
			\le C t^{-1} \bigl(1 + (At)^2\bigr)^{-1/2}
			\left(
			1 + \frac{|x|^3}{t^{3/2} \bigl(1 + (At)^2\bigr)^{3/2}}
			+ \frac{|y|^3}{t^{3/2}}
			\right)^{-1}.
		\end{equation}
	\end{lemma}
	\begin{proof}
		Differentiating (\ref{3.2}) with respect to $\xi$ and $\eta$, we obtain
		\begin{equation}\label{3001}
			\partial_t (\partial_\eta \hat{\mathbb{G}}_2)
			- A\xi \partial_\eta^2 \hat{\mathbb{G}}_2
			- \overline{B}(\xi,\eta) \partial_\eta \hat{\mathbb{G}}_2
			=
			\begin{pmatrix}
				-2\eta & 2\eta \\
				0 & -2\eta
			\end{pmatrix}
			\hat{\mathbb{G}}_2,
		\end{equation}
		\begin{equation}\label{3002}
			\partial_t (\partial_\xi \hat{\mathbb{G}}_2)
			- A\xi \partial_\eta \partial_\xi \hat{\mathbb{G}}_2
			- \overline{B}(\xi,\eta) \partial_\xi \hat{\mathbb{G}}_2
			=
			A \partial_\eta \hat{\mathbb{G}}_2
			+
			\begin{pmatrix}
				-2\xi & 2\xi \\
				0 & -2\xi
			\end{pmatrix}
			\hat{\mathbb{G}}_2,
		\end{equation}
		\begin{equation}\label{3003}
			\partial_t (\partial_\eta^2 \hat{\mathbb{G}}_2)
			- A\xi \partial_\eta^3 \hat{\mathbb{G}}_2
			- \overline{B}(\xi,\eta) \partial_\eta^2 \hat{\mathbb{G}}_2,
			=
			\begin{pmatrix}
				-2\eta & 2\eta \\
				0 & 2\eta
			\end{pmatrix}
			\partial_\eta \hat{\mathbb{G}}_2,
		\end{equation}
		\begin{equation}\label{3004}
			\partial_t (\partial_\xi^2 \hat{\mathbb{G}}_2)
			- A\xi \partial_\eta \partial_\xi^2 \hat{\mathbb{G}}_2
			- \overline{B}(\xi,\eta) \partial_\xi^2 \hat{\mathbb{G}}_2
			=
			A \partial_\eta \partial_\xi \hat{\mathbb{G}}_2
			+
			\begin{pmatrix}
				-2\xi & 2\xi \\
				0 & -2\xi
			\end{pmatrix}
			\partial_\xi \hat{\mathbb{G}}_2.
		\end{equation}
		From (\ref{3001}), using Duhamel's principle and Lemma \ref{lem:3.2}, it follows that
		\begin{equation}\label{3005}
			\begin{aligned}
				\|\partial_\eta \hat{\mathbb{G}}_2(t)\|_{L^1}
				&\leq
				\big\| \hat{\mathbb{G}}_2(t) \cdot \partial_\eta \hat{\mathbb{G}}_2 \big|_{t=0} \big\|_{L^1}
				+ C\int_0^t
				\big\| \hat{\mathbb{G}}_2(t-s) \cdot \eta \hat{\mathbb{G}}_2(s) \big\|_{L^1} ds
				\\
				&\leq
				C\int_0^{\frac{t}{2}}
				\|\eta \hat{\mathbb{G}}_2(t-s)\|_{L^1}
				\|\hat{\mathbb{G}}_2(s)\|_{L^\infty} ds
				+ C\int_{\frac{t}{2}}^t
				\|\hat{\mathbb{G}}_2(t-s)\|_{L^\infty}
				\|\eta \hat{\mathbb{G}}_2(s)\|_{L^1} ds
				\\
				&\leq
				C\int_0^{\frac{t}{2}}
				(t-s)^{-\frac{3}{2}}
				\big(1+(A(t-s))^2\big)^{-\frac{1}{2}} ds
				+ C\int_{\frac{t}{2}}^t
				s^{-\frac{3}{2}}
				\big(1+(As)^2\big)^{-\frac{1}{2}} ds
				\\
				&\leq
				C\, t^{-\frac{1}{2}}
				\big(1+(At)^2\big)^{-\frac{1}{2}},
			\end{aligned}
		\end{equation}
		and
		\begin{equation}\label{3006}
			\begin{aligned}
				\|\partial_\eta \hat{\mathbb{G}}_2(t)\|_{L^\infty}
				&\leq
				\int_0^t
				\big\| \hat{\mathbb{G}}_2(t-s) \cdot 2\eta \hat{\mathbb{G}}_2(s) \big\|_{L^\infty} ds
				\\
				&\leq
				C \Bigg(
				\int_0^{\frac{t}{2}}
				\|\eta \hat{\mathbb{G}}_2(t-s)\|_{L^\infty}
				\|\hat{\mathbb{G}}_2(s)\|_{L^\infty} ds
				+
				\int_{\frac{t}{2}}^t
				\|\hat{\mathbb{G}}_2(t-s)\|_{L^\infty}
				\|\eta \hat{\mathbb{G}}_2(s)\|_{L^\infty} ds
				\Bigg)
				\\
				&\leq
				C \Bigg(
				\int_0^{\frac{t}{2}} (t-s)^{-\frac{1}{2}} ds
				+
				\int_{\frac{t}{2}}^t s^{-\frac{1}{2}} ds
				\Bigg)
				\\
				&\leq
				C \sqrt{t}.
			\end{aligned}
		\end{equation}
		By (\ref{3002}) , we have
		\[
		\begin{aligned}
			\|\partial_\xi \hat{\mathbb{G}}_2(t)\|_{L^1}
			& \leq \left\|\hat{\mathbb{G}}_2(t)\cdot \partial_\xi \hat{\mathbb{G}}_2\big|_{t=0}\right\|_{L^1}
			+ C\int_0^t \left\|\hat{\mathbb{G}}_2(t-s)\, \partial_\eta \hat{\mathbb{G}}_2(s)\right\|_{L^1} ds \\
			& \quad + C \int_0^t \left\|\hat{\mathbb{G}}_2(t-s)\, \xi \hat{\mathbb{G}}_2(s)\right\|_{L^1} ds \\
			& \leq C\int_0^{\frac{t}{2}} \left\|\hat{\mathbb{G}}_2(t-s)\right\|_{L^1} \left\|\partial_\eta \hat{\mathbb{G}}_2(s)\right\|_{L^\infty} ds
			+ C\int_{\frac{t}{2}}^t \left\|\hat{\mathbb{G}}_2(t-s)\right\|_{L^\infty} \left\|\partial_\eta \hat{\mathbb{G}}_2(s)\right\|_{L^1} ds \\
			& \quad + C\int_0^{\frac{t}{2}} \left\|\xi\hat{\mathbb{G}}_2(t-s)\right\|_{L^1} \left\|\hat{\mathbb{G}}_2(s)\right\|_{L^\infty} ds
			+ C\int_{\frac{t}{2}}^t \left\|\hat{\mathbb{G}}_2(t-s)\right\|_{L^\infty} \left\|\xi\hat{\mathbb{G}}_2(s)\right\|_{L^1} ds \\
			& \eqqcolon \sum_{i=1}^4 I_i.
		\end{aligned}
		\]
		Using Duhamel's principle, Lemma \ref{lem:3.2}, (\ref{3005}) and (\ref{3006}), we can compute that
		\[
		\begin{aligned}
			I_1 & \leq C \int_0^{\frac{t}{2}} (t-s)^{-1} \left(1+(A(t-s))^2\right)^{-\frac{1}{2}} s^{\frac{1}{2}} ds \\
			& \leq C t^{-1} \left(1+(At)^2\right)^{-\frac{1}{2}} t^{\frac{3}{2}}
			\leq C t^{-\frac{1}{2}},
		\end{aligned}
		\]
		
		\[
		\begin{aligned}
			I_2 & \leq C \int_{\frac{t}{2}}^t s^{-\frac{1}{2}} \left(1+(As)^2\right)^{-\frac{1}{2}} ds \\
			& \leq C \left(1+(At)^2\right)^{-\frac{1}{2}} t^{\frac{1}{2}}
			\leq C t^{-\frac{1}{2}},
		\end{aligned}
		\]
		
		\[
		\begin{aligned}
			I_3 & \leq C \int_0^{\frac{t}{2}} (t-s)^{-\frac{3}{2}} \left(1+(A(t-s))^2\right)^{-1} ds \\
			& \leq C t^{-\frac{1}{2}} \left(1+(At)^2\right)^{-1},
		\end{aligned}
		\]
		
		\[
		\begin{aligned}
			I_4 & \leq C \int_{\frac{t}{2}}^t s^{-\frac{3}{2}} \left(1+(As)^2\right)^{-1} ds \\
			& \leq C t^{-\frac{1}{2}} \left(1+(At)^2\right)^{-1}.
		\end{aligned}
		\]
		Combining each estimate of $I_I$, we obtain
		\begin{equation}\label{3.23}
			\|\partial_\xi \hat{\mathbb{G}}_2(t)\|_{L^1}\leq Ct^{-\frac{1}{2}}.
		\end{equation}
		For $\|\partial_\xi \hat{\mathbb{G}}_2(t)\|_{L^\infty}$, we have
		\begin{equation}\label{3.24}
			\begin{aligned}
				\|\partial_\xi \hat{\mathbb{G}}_2(t)\|_{L^\infty}
				& \leq \int_0^t \left\|\hat{\mathbb{G}}_2(t-s)\right\|_{L^\infty} \left\|\partial_\eta \hat{\mathbb{G}}_2(s)\right\|_{L^\infty} ds
				+ \int_0^t \left\|\hat{\mathbb{G}}_2(t-s)\right\|_{L^\infty} \left\|\xi\hat{\mathbb{G}}_2(s)\right\|_{L^\infty} ds \\
				& \leq C\int_0^t s^{\frac{1}{2}} ds + C\int_0^t s^{-\frac{1}{2}} \left(1+(As)^2\right)^{-\frac{1}{2}} ds \\
				& \leq C\left(t^{\frac{3}{2}}+t^{\frac{1}{2}}\right).
			\end{aligned}
		\end{equation}
		
		By (\ref{3003}) and Duhamel's principle, we have
		\begin{equation*}
			\begin{aligned}
				\|\partial_\eta^2 \hat{\mathbb{G}}_2(t)\|_{L^1} 
				&\leq \left\|\hat{\mathbb{G}}_2(t)\cdot \partial_\eta^2 \hat{\mathbb{G}}_2\big|_{t=0}\right\|_{L^1} 
				+ \int_0^t \left\|\hat{\mathbb{G}}_2(t-s)\, \eta\partial_\eta \hat{\mathbb{G}}_2(s)\right\|_{L^1} ds\\
				&\leq \int_0^{\frac{t}{2}} \left\|\eta\hat{\mathbb{G}}_2(t-s)\right\|_{L^1}\, \left\|\partial_\eta \hat{\mathbb{G}}_2(s)\right\|_{L^\infty} ds
				+ \int_{\frac{t}{2}}^t \left\|\eta\hat{\mathbb{G}}_2(t-s)\right\|_{L^\infty}\, \left\|\partial_\eta \hat{\mathbb{G}}_2(s)\right\|_{L^1} ds\\
				&\eqqcolon J_1 + J_2.
			\end{aligned}
		\end{equation*}
		(\ref{3005}), (\ref{3006}) and Lemma \ref{lem:3.2} give that
		\begin{equation*}
			\begin{aligned}
				J_1& \leq C \int_0^{\frac{t}{2}} (t-s)^{-\frac{3}{2}} \left(1+(A(t-s))^2\right)^{-\frac{1}{2}} s^{\frac{1}{2}} ds\\
				&\leq C t^{-\frac{3}{2}} \left(1+(At)^2\right)^{-\frac{1}{2}} t^{\frac{3}{2}}
				= C \left(1+(At)^2\right)^{-\frac{1}{2}},\\
				J_2 &\leq C \int_{\frac{t}{2}}^t (t-s)^{-\frac{1}{2}} s^{-\frac{1}{2}} \left(1+(As)^2\right)^{-\frac{1}{2}} ds\\
				&\leq C \left(1+(At)^2\right)^{-\frac{1}{2}} t^{-\frac{1}{2}} t^{\frac{1}{2}}
				= C \left(1+(At)^2\right)^{-\frac{1}{2}}.
			\end{aligned}
		\end{equation*}
		
		Thus
		\begin{equation}\label{3.26}
			\|\partial_\eta^2 \hat{\mathbb{G}}_2(t)\|_{L^1} \leq C \left(1+(At)^2\right)^{-\frac{1}{2}}.
		\end{equation}
		
		Regarding $\|\partial_\eta^2 \hat{\mathbb{G}}_2(t)\|_{L^\infty}$, we have
		\begin{equation}\label{3.27}
			\begin{aligned}
				\|\partial_\eta^2 \hat{\mathbb{G}}_2(t)\|_{L^\infty}
				& \leq \int_0^t \left\|\eta \hat{\mathbb{G}}_2(t-s)\right\|_{L^\infty}
				\left\|\partial_\eta \hat{\mathbb{G}}_2(s)\right\|_{L^\infty} ds \\
				& \leq C \int_0^t (t-s)^{-\frac{1}{2}} s^{\frac{1}{2}} ds \\
				& \leq C t.
			\end{aligned}
		\end{equation}
		
		In view of (\ref{3004}), to obtain the estimate of $\partial_\xi^2 \hat{\mathbb{G}}_2$,
		we first consider the estimate of $\partial_\xi \partial_\eta \hat{\mathbb{G}}_2$.
		
		Differentiating both sides of (\ref{3002}) with respect to $\eta$, we get
		\begin{equation}\label{3.026}
			\begin{aligned}
				&\partial_t \partial_\xi \partial_\eta \hat{\mathbb{G}}_2
				- A \partial_\eta\bigl(\partial_\xi \partial_\eta \hat{\mathbb{G}}_2\bigr)
				- \bar{B}(\xi,\eta)\, \partial_\xi \partial_\eta \hat{\mathbb{G}}_2\\
				={} &\begin{pmatrix} -2\eta & 2\eta \\ 0 & 2\eta \end{pmatrix} \partial_\xi \hat{\mathbb{G}}_2
				+ A \partial_\eta^2 \hat{\mathbb{G}}_2
				+ \begin{pmatrix} -2\xi & 2\xi \\ 0 & 2\xi \end{pmatrix} \partial_\eta \hat{\mathbb{G}}_2.
			\end{aligned}
		\end{equation}
		By (\ref{3.026}) and Duhamel's principle, we have
		\[
		\begin{aligned}
			\|\partial_\xi \partial_\eta \hat{\mathbb{G}}_2\|_{L^1}
			& \leq C\int_0^t \left\|\hat{\mathbb{G}}_2(t-s)\, \eta\, \partial_\xi \hat{\mathbb{G}}_2(s)\right\|_{L^1} ds+ C\int_0^t \left\|\hat{\mathbb{G}}_2(t-s)\, \partial_\eta^2 \hat{\mathbb{G}}_2(s)\right\|_{L^1} ds \\
			& +C \int_0^t \left\|\hat{\mathbb{G}}_2(t-s)\, \xi\, \partial_\eta \hat{\mathbb{G}}_2(s)\right\|_{L^1} ds \\
			& \eqqcolon \sum_{k=1}^3 K_i.
		\end{aligned}
		\]
		Next, we estimate $K_i$ separately. For $K_1$, by using Lemma \ref{lem:3.2}, (\ref{3.23}) and (\ref{3.24}), we have
		\begin{align*}
			K_1 &\le C\int_{0}^{\frac{t}{2}} \left\| \eta \hat{\mathbb{G}}_2(t-s) \right\|_{L^1} \left\| \partial_\xi \hat{\mathbb{G}}_2(s) \right\|_{L^\infty} ds
			+C \int_{\frac{t}{2}}^{t} \left\| \eta \hat{\mathbb{G}}_2(t-s) \right\|_{L^\infty} \left\| \partial_\xi \hat{\mathbb{G}}_2(s) \right\|_{L^1} ds \\
			&\le C \int_{0}^{\frac{t}{2}} (t-s)^{-\frac{3}{2}} \left(1 + (A(t-s))^2\right)^{-\frac{1}{2}} \left(s^{\frac{3}{2}} + s^{\frac{1}{2}}\right) ds
			+ C \int_{\frac{t}{2}}^{t} (t-s)^{-\frac{1}{2}} s^{-\frac{1}{2}} ds \\
			&\le C t^{-\frac{3}{2}} \left(1 + (At)^2\right)^{-\frac{1}{2}} \left(t^{\frac{5}{2}} + t^{\frac{3}{2}}\right) + C t^{-\frac{1}{2}} t^{\frac{1}{2}} \\
			&\le C.
		\end{align*}
		For $K_2$, Lemma \ref{lem:3.2}, (\ref{3.26}) and (\ref{3.27}) give that
		\begin{align*}
			K_2 &\le \int_{0}^{\frac{t}{2}} \left\| \hat{\mathbb{G}}_2(t-s) \right\|_{L^1} \left\| \partial_\eta^2 \hat{\mathbb{G}}_2(s) \right\|_{L^\infty} ds
			+ \int_{\frac{t}{2}}^{t} \left\| \hat{\mathbb{G}}_2(t-s) \right\|_{L^\infty} \left\| \partial_\eta^2 \hat{\mathbb{G}}_2(s) \right\|_{L^1} ds \\
			&\le C \int_{0}^{\frac{t}{2}} (t-s)^{-1} \left(1 + (A(t-s))^2\right)^{-\frac{1}{2}} s \, ds
			+ C \int_{\frac{t}{2}}^{t} \left(1 + (As)^2\right)^{-\frac{1}{2}} ds \\
			&\le C.
		\end{align*}
		For $K_3$, from Lemma \ref{lem:3.2}, (\ref{3005}) and (\ref{3006}), we can obtain that
		\begin{align*}
			K_3 &\le \int_{0}^{\frac{t}{2}} \left\| \xi \hat{\mathbb{G}}_2(t-s) \right\|_{L^1} \left\| \partial_\eta \hat{\mathbb{G}}_2(s) \right\|_{L^\infty} ds
			+ \int_{\frac{t}{2}}^{t} \left\| \xi \hat{\mathbb{G}}_2(t-s) \right\|_{L^\infty} \left\| \partial_\eta \hat{\mathbb{G}}_2(s) \right\|_{L^1} ds \\
			&\le C \int_{0}^{\frac{t}{2}} (t-s)^{-\frac{3}{2}} \left(1 + (A(t-s))^2\right)^{-\frac{1}{2}} s^{\frac{1}{2}} ds \\
			&\quad + C \int_{\frac{t}{2}}^{t} (t-s)^{-\frac{1}{2}} \left(1 + (A(t-s))^2\right)^{-\frac{1}{2}} s^{-\frac{1}{2}} \left(1 + (As)^2\right)^{-\frac{1}{2}} ds \\
			&\le C.
		\end{align*}
		So far, we obtain the following estimate 
		\begin{equation}\label{3.29}
			\|\partial_\xi \partial_\eta \hat{\mathbb{G}}_2\|_{L^1}\leq C.
		\end{equation}
		Also, we have
		\begin{align*}
			\left\| \partial_\xi \partial_\eta \hat{\mathbb{G}}_2(t) \right\|_{L^\infty}
			& \le \int_{0}^{t} \left\| \eta \hat{\mathbb{G}}_2(t-s) \right\|_{L^\infty}
			\left\| \partial_\xi \hat{\mathbb{G}}_2(s) \right\|_{L^\infty} ds \\
			& \quad + \int_{0}^{t} \left\| \hat{\mathbb{G}}_2(t-s) \right\|_{L^\infty}
			\left\| \partial_\eta^2 \hat{\mathbb{G}}_2(s) \right\|_{L^\infty} ds \\
			& \quad + \int_{0}^{t} \left\| \xi \hat{\mathbb{G}}_2(t-s) \right\|_{L^\infty}
			\left\| \partial_\eta \hat{\mathbb{G}}_2(s) \right\|_{L^\infty} ds \\
			& \eqqcolon \sum_{i=1}^3 J_i.
		\end{align*}
		By using Lemma \ref{lem:3.2}, (\ref{3006}), (\ref{3.24}) and (\ref{3.27}), we have
		\begin{align*}
			J_1 &\le \int_{0}^{t} (t-s)^{-\frac{1}{2}} \left(s^{\frac{3}{2}} + s^{\frac{1}{2}}\right) ds
			\le C\left(t^2 + t\right), \\
			J_2 &\le C \int_{0}^{t} s \, ds \le C t^2, \\
			J_3 &\le \int_{0}^{t} (t-s)^{\frac{1}{2}} \left(1 + (A(t-s))^2\right)^{-\frac{1}{2}} s^{\frac{1}{2}} ds
			\le C t.
		\end{align*}
		Hence,
		\begin{equation}\label{3.30}
			\left\| \partial_\xi \partial_\eta \hat{\mathbb{G}}_2(t) \right\|_{L^\infty} \le C\left(t^2 + t\right).
		\end{equation}
		From (\ref{3004}), by Duhamel's principle, we have
		\begin{align*}
			\left\| \partial_\xi^2\hat{\mathbb{G}}_2(t) \right\|_{L^1}
			& \le \left\| \hat{\mathbb{G}}_2(t) \cdot \partial_\xi^2 \hat{\mathbb{G}}_2 \big|_{t=0} \right\|_{L^1}
			+ \int_{0}^{t} \left\| \hat{\mathbb{G}}_2(t-s) \xi\partial_\xi \hat{\mathbb{G}}_2(s) \right\|_{L^1} ds \\
			& \quad + \left( \int_{0}^{t} \left\| \hat{\mathbb{G}}_2(t-s) \partial_\xi \partial_\eta \hat{\mathbb{G}}_2(s) \right\|_{L^1} ds \right) \\
			& \le \int_{0}^{t} \left\| \hat{\mathbb{G}}_2(t-s) \xi\partial_\xi \hat{\mathbb{G}}_2(s) \right\|_{L^1} ds
			+ C \int_{0}^{t} \left\| \hat{\mathbb{G}}_2(t-s) \partial_\xi \partial_\eta \hat{\mathbb{G}}_2(s) \right\|_{L^1} ds \\
			& \eqqcolon \sum_{i=1}^{2} M_i.
		\end{align*}
		Using Lemma \ref{lem:3.2}, (\ref{3.23}), (\ref{3.24}), (\ref{3.29}) and (\ref{3.30}), we can have
		\begin{align*}
			M_1 &\le \int_{0}^{\frac{t}{2}} \left\| \partial \hat{\mathbb{G}}_2(t-s) \right\|_{L^1} \left\| \partial_\xi \hat{\mathbb{G}}_2(s) \right\|_{L^\infty} ds
			+ \int_{\frac{t}{2}}^{t} \left\| \partial \hat{\mathbb{G}}_2(t-s) \right\|_{L^\infty} \left\| \partial_\xi \hat{\mathbb{G}}_2(s) \right\|_{L^1} ds \\
			&\le \int_{0}^{\frac{t}{2}} (t-s)^{-\frac{3}{2}} \left(1 + (A(t-s))^2\right)^{-1} \left(s^{\frac{1}{2}} + s^{\frac{3}{2}}\right) ds
			+ \int_{\frac{t}{2}}^{t} (t-s)^{-\frac{1}{2}} \left(1 + (A t - s)^2\right)^{-\frac{1}{2}} s^{-\frac{1}{2}} ds \\
			&\le C \left(1 + (A t)^2\right)^{\frac{1}{2}}, \\[6pt]
			M_2 &\le C \int_{0}^{\frac{t}{2}} \left\| \hat{\mathbb{G}}_2(t-s) \right\|_{L^1} \left\| \partial_\xi \partial_\eta \hat{\mathbb{G}}_2(s) \right\|_{L^\infty} ds
			+ C \int_{\frac{t}{2}}^{t} \left\| \hat{\mathbb{G}}_2(t-s) \right\|_{L^\infty} \left\| \partial_\xi \partial_\eta \hat{\mathbb{G}}_2(s) \right\|_{L^1} ds \\
			&\le C \int_{0}^{\frac{t}{2}} (t-s)^{-1} \left(1 + (A(t-s))^2\right)^{-\frac{1}{2}} (t^2 + t) ds
			+ C \int_{\frac{t}{2}}^{t} 1 \, ds \\
			&\le C t \le C \left(1 + (A t)^2\right)^{\frac{1}{2}}.
		\end{align*}
		So we can conclude that 
		\begin{equation}\label{3.303}
			\left\| \partial_\xi^2\hat{\mathbb{G}}_2(t) \right\|_{L^1}\le C \left(1 + (A t)^2\right)^{\frac{1}{2}}.
		\end{equation}
		Differentiating both sides of (\ref{3003}) with respect to $\eta$, we obtain
		\begin{equation}\label{3.31}
			\partial_t (\partial_\eta^3 \hat{\mathbb{G}}_2)
			- A \partial_\eta (\partial_\eta^3 \hat{\mathbb{G}}_2)
			- \bar B(s,\eta) (\partial_\eta^3 \hat{\mathbb{G}}_2)
			= 2
			\begin{pmatrix}
				-2\eta & 2\eta \\
				0 & -2\eta
			\end{pmatrix}
			\partial_\eta^2 \hat{\mathbb{G}}_2.
		\end{equation}
		From (\ref{3.31}), we have
		\begin{align*}
			\|\partial_\eta^3 \hat{\mathbb{G}}_2(t)\|_{L^1}
			& \le \left. \left\| \hat{\mathbb{G}}_2(t) \cdot \partial_\eta^3 \hat{\mathbb{G}}_2 \right|_{t=0} \right\|_{L^1}
			+ C \int_0^t \left\| \hat{\mathbb{G}}_2(t-s) \cdot \eta \cdot \partial_\eta^2 \hat{\mathbb{G}}_2(s) \right\|_{L^1} ds \\
			& \le C \int_0^{t/2} \left\| \eta \hat{\mathbb{G}}_2(t-s) \right\|_{L^1}
			\left\| \partial_\eta^2 \hat{\mathbb{G}}_2(s) \right\|_{L^\infty} ds \\
			& \quad + C \int_{t/2}^t \left\| \eta \hat{\mathbb{G}}_2(t-s) \right\|_{L^\infty}
			\left\| \partial_\eta^2 \hat{\mathbb{G}}_2(t-s) \right\|_{L^1} ds \\
			& \eqqcolon F_1 + F_2.
		\end{align*}
		Using Lemma \ref{lem:3.2}, (\ref{3.26}) and (\ref{3.27}), we get
		\begin{align*}
			F_1 & \le C \int_0^{t/2} (t-s)^{-3/2} \left(1 + (A(t-s))^2\right)^{-1/2} s^{1/2} ds \\
			& \le C \left(1 + (At)^2\right)^{-1/2} t^{1/2},
		\end{align*}
		and
		\begin{align*}
			F_2 & \le C \int_{t/2}^t (t-s)^{-1/2} \left(1 + (As)^2\right)^{-1/2} ds \\
			& \le C \left(1 + (At)^2\right)^{-1/2} t^{1/2}.
		\end{align*}
		
		Therefore, we conclude that
		\begin{equation}\label{3.32}
			\left\| \partial_\eta^3 \hat{\mathbb{G}}_2(t) \right\|_{L^1}
			\le C \left(1 + (At)^2\right)^{-1/2} t^{1/2}.
		\end{equation}
		Differentiating both sides of equation (\ref{3004}) with respect to $\xi$, we obtain
		\begin{align*}
			\partial_t (\partial_\xi^3 \hat{\mathbb{G}}_2)
			- A \partial_\eta (\partial_\xi^3 \hat{\mathbb{G}}_2)
			- \bar{B}(\xi,\eta) (\partial_\xi^3 \hat{\mathbb{G}}_2)
			&= 2A \partial_\eta (\partial_\xi^2 \hat{\mathbb{G}}_2)
			+ 2
			\begin{pmatrix}
				-2\xi & 2\xi \\
				0 & -2\xi
			\end{pmatrix}
			\partial_\xi^2 \hat{\mathbb{G}}_2.
		\end{align*}
		Simple calculation shows that
		\begin{equation}\label{3.33}
			\left\| \partial_\xi^3 \hat{\mathbb{G}}_2(t) \right\|_{L^1}
			\le C t^{\frac{1}{2}} \left(1 + (At)^2\right).
		\end{equation}
		Hence, by the definition of the Fourier transform, Lemma \ref{lem:3.2}, (\ref{3.32}) and (\ref{3.33}), we have
		\begin{equation}\label{666}
			\begin{aligned}
				|\mathbb{G}_2(x,y,t)| 
				&= \left| \int_{\mathbb{R}^2} e^{i(x\xi + y\eta)} \, \hat{\mathbb{G}}_2(\xi,\eta,t) \, d\xi d\eta \right|  \leq Ct^{-1}(1+(At)^2)^{-\frac{1}{2}},\\
				|x^3 \mathbb{G}_2(x,y,t)|
				&= \left| \int_{\mathbb{R}^2} e^{i(x\xi + y\eta)} \, \partial_\xi^3 \left( \hat{\mathbb{G}}_2(\xi,\eta,t) \right) d\xi d\eta \right| \leq t^{\frac{1}{2}}(1+(At)^2),\\
				|y^3 \mathbb{G}_2(x,y,t)| 
				&= \left| \int_{\mathbb{R}^2} e^{i(x\xi + y\eta)} \, \partial_\eta^3 \left( \hat{\mathbb{G}}_2(\xi,\eta,t) \right) d\xi d\eta \right| \leq C \leq t^{\frac{3}{2}}t^{-1}(1+(At)^2)^{-\frac{1}{2}}.
			\end{aligned}
		\end{equation}
		Dividing $\mathbb{R}^2$ into 4 areas, i.e, $\mathbb{R}^2=W_1+W_2+W_3+W_4$, and they are defined as:
		\begin{equation*}
			\begin{array}{rl}
				W_1=&\{|x|^3\leq t^{\frac{3}{2}}(1+(At)^2)^{\frac{3}{2}}, |y|^3\leq t^{\frac{3}{2}}\},\\ W_2=&\{|x|^3\leq t^{\frac{3}{2}}(1+(At)^2)^{\frac{3}{2}}, |y|^3> t^{\frac{3}{2}}\},\\
				W_3=&\{|x|^3>t^{\frac{3}{2}}(1+(At)^2)^{\frac{3}{2}}, |y|^3\leq t^{\frac{3}{2}}\},\\ W_4=&\{|x|^3>t^{\frac{3}{2}}(1+(At)^2)^{\frac{3}{2}}, |y|^3>t^{\frac{3}{2}}\}.
			\end{array}
		\end{equation*}
		Then, we have
		\begin{equation}\label{eq:2.22}
			\begin{array}{rl}
				1+\frac{|x|^3}{t^{\frac{3}{2}}(1+(At)^2)^{\frac{3}{2}}}+\frac{|y|^3}{t^{\frac{3}{2}}}
				\leq 3\begin{cases}
					1,\ \ (x, y)\in W_1,\\
					\frac{|y|^3}{t^{\frac{3}{2}}},\ \ (x, y)\in W_2,\\
					\frac{|x|^3}{t^{\frac{3}{2}}(1+(At)^2)^{\frac{3}{2}}},\ \ (x, y)\in W_3,\\
					\frac{|x|^3}{t^{\frac{3}{2}}(1+(At)^2)^{\frac{3}{2}}}+\frac{|y|^3}{t^{\frac{3}{2}}},\ \ (x, y)\in W_4.
				\end{cases}
			\end{array}
		\end{equation}
		Combining (\ref{666}) and \eqref{eq:2.22}, we can obtain \eqref{777}. Then  we proved the lemma
	\end{proof}
	\begin{lemma}\label{lem:3.6}
		For the first-order derivatives of \(\mathbb{G}_2\), we have the following pointwise estimate:
		\[
		|\partial_x \mathbb{G}_2(x,y,t)|
		\le C t^{-\frac{3}{2}} \bigl(1 + (At)^2\bigr)^{-1}
		\left(
		1 + \frac{|x|^3}{t^{\frac{3}{2}} \bigl(1 + (At)^2\bigr)^{\frac{3}{2}}}
		+ \frac{|y|^3}{t^{\frac{3}{2}}}
		\right)^{-1},
		\]
		\[
		|\partial_y \mathbb{G}_2(x,y,t)|
		\le C t^{-\frac{3}{2}} \bigl(1 + (At)^2\bigr)^{-\frac{1}{2}}
		\left(
		1 + \frac{|x|^3}{t^{\frac{3}{2}} \bigl(1 + (At)^2\bigr)^{\frac{3}{2}}}
		+ \frac{|y|^3}{t^{\frac{3}{2}}}
		\right)^{-1}.
		\]
	\end{lemma}
	
	\begin{proof}
		The proof is similar to Lemma \ref{lem:3.5}, so we omit it.
	\end{proof}
	For the convenience of later writing, we also introduce notations:
	$$
	\mathbb{G}(x,y,t)\circledast f=\int_{\mathbb{R}^2}\mathbb{G}(x-x^\prime, y,t;y^\prime)f(x^\prime, y^\prime)dx^\prime dy^\prime.
	$$
	and  we note
	$$
	\interleave\mathbb{G}(t)\interleave_{L^p}=\max \{\|\mathbb{G}(\cdot-x^\prime, \cdot,t;y^\prime)\|_{L^p},\ \|\mathbb{G}(x-\cdot, y,t;\cdot)\|_{L^p}\}.
	$$
	\begin{proposition}\label{prop3.1}
		For any non-negative integers \(k = k_1 + k_2\) and \(p \ge 2\),
		\begin{equation}\label{03.1}
			\interleave\partial^{k_1}_{x}\partial^{k_2}_{y}\mathbb{G}\interleave_{L^p}
			\le C t^{-\left(1-\frac{1}{p}\right)-\frac{k}{2}} \bigl(1 + (At)^2\bigr)^{-\frac{1}{2}\left(1-\frac{1}{p} + k_1\right)},
		\end{equation}
		\begin{equation}\label{03.2}
			\interleave\partial^{k_1}_{x'}\partial^{k_2}_{y'}\mathbb{G}\interleave_{L^p}
			\le C t^{-\left(1-\frac{1}{p}\right)-\frac{k}{2}} \bigl(1 + (At)^2\bigr)^{-\frac{1}{2}\left(1-\frac{1}{p} + k_1\right)}.
		\end{equation}
		Moreover, when \(k_1 = 0, k_2 = 1\) or \(k_1 = 1, k_2 = 0\) or \(k_1 = 0, k_2 = 0\)  , the above inequalities hold for all \(p \ge 1\).
	\end{proposition}
	
	\begin{proof}
		The estimate (\ref{03.1}) follows from Lemma \ref{lem:3.3}.
		Using (\ref{3.17}) and (\ref{03.1}), we have (\ref{03.2}) is also true.
		As for $1\leq p<2$, we consider the estimate for $\|\partial_x \mathbb{G}(\cdot, \cdot, t; x^\prime,y^\prime) \|_{L^p}$, $\|\partial_y \mathbb{G}(\cdot, \cdot, t; x^\prime,y^\prime)\|_{L^p}$ and $\|\mathbb{G}(\cdot, \cdot, t; x^\prime,y^\prime)\|_{L^p}$. By using lemma \ref{lem:3.5} and Lemma \ref{lem:3.6}, we obtain that
		\[
		\|\partial_x \mathbb{G}(\cdot, \cdot, t; x^\prime,y^\prime)\|_{L^p}
		\le C \|\partial_x \mathbb{G}_2(\cdot, \cdot, t)\|_{L^p}
		\le C t^{-\frac{1}{2}-\left(1-\frac{1}{p}\right)} \bigl(1 + (At)^2\bigr)^{-\frac{1}{2}\left(1+1-\frac{1}{p}\right)},
		\]
		\[
		\|\partial_y \mathbb{G}(\cdot, \cdot, t; x^\prime,y^\prime)\|_{L^p}
		\le C \|\partial_y \mathbb{G}_2(\cdot, \cdot, t)\|_{L^p}
		\le C t^{-\frac{1}{2}-\left(1-\frac{1}{p}\right)} \bigl(1 + (At)^2\bigr)^{-\frac{1}{2}\left(1-\frac{1}{p}\right)}.
		\]
		\[
		\|\mathbb{G}(\cdot, \cdot, t; x^\prime,y^\prime)\|_{L^p}
		\le C \|\mathbb{G}_2(\cdot, \cdot, t)\|_{L^p}
		\le C t^{-\left(1-\frac{1}{p}\right)} \bigl(1 + (At)^2\bigr)^{\left(1+1-\frac{1}{p}\right)},
		\]
		By using \eqref{eq:2.14} and \eqref{3.17} again, we can obtaion (\ref{03.1}) and \eqref{03.2}, then we proved the Proposition.
	\end{proof}
	
	\section{Nonlinear Stability}\label{s:NonlinearG}
	In this section, after briefly stating the local existence results of the solution, we will focus on discussing how to use the estimate of the Green's function to obtain the global existence of the solution and the stability of the solution.\par
	
	\subsection{The local existence of the solution}
	The local existence of the system (\ref{1.4})-(\ref{001}) can be obtained using the standard method. For the convenience of the reader, we only give the theorem here, and do not describe the proof.
	\begin{theorem}\label{thm3.1}
		Suppose  that $m_0, \omega_0\in L^1(\mathbb{R}^2)\cap L^{\infty}(\mathbb{R}^2)$, and take $\nu = \kappa =\frac{1}{2} \mu$ which satisfies $0<\mu \ll 1$, and define $A=\dfrac1\mu$. There exists constants $\mu_0$  and $c_0$ independent  of $\mu$  so that if 
		\begin{equation*}
			\|(m_0, \omega_{0})\|_{L^1\cap L^{\infty}}\leq c_0A^{-\frac{3}{4}}
		\end{equation*}
		for some sufficiently small $c_0,0<\mu\leq \mu_0,$, there exists a positive constant $T_0=T((m_0, \omega_{0}))$, when $t\leq T_0$, the equation
		(\ref{1.4})-(\ref{001})has a unique local classical solution $(m(t, x, y), w(t, x, y))$ which satisfies
		\begin{equation*}
			(m(t, x, y), w(t, x, y))\in C([0,T_0], L^{\frac{4}{3}}(\mathbb{R}^2)\cap L^{\infty}(\mathbb{R}^2)),
		\end{equation*}
		and for a given time $T_0$, if the solution of the equation satisfies the condition
		\begin{equation*}
			\lim_{t\rightarrow T_0}\sup_{0\leq s\leq T_0}\|(m(s, \cdot), w(s, \cdot))\|_{L^\infty(\mathbb{R}^2)}< \infty,
		\end{equation*}
		then there exists a sufficiently small $\varepsilon>0$, time $T_0$ can be extended to $T_0+\varepsilon$.
	\end{theorem}
	
	\subsection{Nonlinear stability}
	In the following, we will mainly consider the global existence of the solution and the stability of the solution under the condition
	\begin{equation}\label{4.1}
		\|(m_0, \omega_{0})\|_{L^1\cap L^{\infty}}=\epsilon\leq c_0A^{-\frac{3}{4}}
	\end{equation}
	By using the bootstrap argument, for any $t \in[0, T]$ and $p\geq \frac{4}{3},$ we begin with the hypothesis that 
	\begin{equation}\label{4.2}
		\mathbf{H}(t):	||(m(t, \cdot,\cdot), \omega(t, \cdot,\cdot))||_{L^p}\leq \delta(1+t)^{-2(1-\frac{1}{p})}\epsilon,
	\end{equation}
	and the conclusion, denoted as
	\begin{equation}\label{4.3}
		\mathbf{C}(t):	||(m(t, \cdot,\cdot), \omega(t, \cdot,\cdot))||_{L^p}\leq \frac{\delta}{2}(1+t)^{-2(1-\frac{1}{p})}\epsilon,
	\end{equation}
	where $\delta>0$ is a fixed number, which is determined in the following text.
	The conditions $(b)-(d)$ stated in Lemma \ref{lem2.2} are satisfied, and we only need to verify condition $(a)$ under the assumption of \eqref{4.1}.\par
	In what follows, we discuss the nonlinear stability in different cases.
	\begin{proposition}\label{pro4.1}
		Assuming that $t\leq2 A^{-\frac{1}{2}}$ and (\ref{4.1}) , (\ref{4.2}) hold, then we can infer$$||(m(t, \cdot,\cdot), \omega(t, \cdot,\cdot))||_{L^p}\leq \frac{\delta}{2}(1+t)^{-2(1-\frac{1}{p})}\epsilon,\ \ \ \\ \forall p\geq \frac{4}{3}.$$
	\end{proposition}
	\begin{proof}
		Let \( V(x,y,t) = \begin{pmatrix} m(x,y,t) \\ w(x,y,t) \end{pmatrix} \),\ \ \( V_0(x,y) = \begin{pmatrix} m_0(x,y) \\ w_0(x,y) \end{pmatrix} \). Then Duhamel's principle and  intergration by parts gives
		\begin{equation}\label{eq:4.5}
			\begin{aligned}
				&V(x,y,t)=\int_{\mathbb{R}^2}\mathbb{G}(x-x',y,t;y')V_{0}(x',y')dx'dy'\\
				&+\int_{0}^{t}\int_{\mathbb{R}^2}-\mathbb{G}(x-x',y,t-s;y')\nabla\cdot\left( uV\right) (x',y',s)dx'dy'ds\\
				&=\int_{\mathbb{R}^2}\mathbb{G}(x-x',y,t;y')V_{0}(x',y')dx'dy'\\
				&+\int_{0}^{t}\int_{\mathbb{R}^2}\nabla_{x',y'}\mathbb{G}(x-x',y,t-s;y')\left( uV\right) (x',y',s)dx'dy'ds\\
				&=\mathbb{G}(t)\circledast V_0+\int_{0}^{t}\nabla_{x',y'}\mathbb{G}(t-s)\circledast(uV)(s)ds\\
				&=\mathbb{G}(t)\circledast V_0+\int_{0}^{\frac{t}{2}}\nabla_{x',y'}\mathbb{G}(t-s)\circledast(uV)(s)ds+\int_{\frac{t}{2}}^t\nabla_{x',y'}\mathbb{G}(t-s)\circledast(uV)(s)ds\\
				&=I_1+I_{21}+I_{22}.
			\end{aligned}
		\end{equation}
		For $I_1$, using Lemma \ref{lem2.1},  on the one hand, we have
		\begin{equation}\label{eq:20}
			\begin{aligned}
				\| I_1\|_{L^p}&=\left\|\mathbb{G}(t)\circledast V_0\right\|_{L^p}\leq \interleave\mathbb{G}(t)\interleave_{L^p}\| V_{0}\|_{L^1}\\
				& \leq CA^{-(1-{\frac{1}{p}})}t^{-2(1-\frac{1}{p})}\|V_{0}\|_{L^1},
			\end{aligned}
		\end{equation}
		on the other hand, 
		\begin{equation}\label{eq:21}
			\begin{aligned}
				\| I_1\|_{L^p}&=\left\|\mathbb{G}(t)\circledast V_0\right\|_{L^p}\leq \interleave\mathbb{G}\interleave_{L^1}\|V_{0}\|_{L^p}\\
				& \leq C\|V_{0}\|_{L^p},
			\end{aligned}
		\end{equation}
		Combining  \eqref{eq:20} and \eqref{eq:21}, we can get
		\begin{equation}\label{4.7}
			\| I_1\|_{L^p}\leq C(1+t)^{-2(1-\frac{1}{p})}\epsilon.
		\end{equation}
		We choose appropriate $\delta$, then it holds
		\begin{equation}\label{4.8}
			\| I_1\|_{L^p}\leq \frac{\delta}{2}(1+t)^{-2(1-\frac{1}{p})}\epsilon.
		\end{equation}	
		For $I_{21}$, on the one hand, using Lemma \ref{lem2.1}, Lemma \ref{lem2.3}, Propsition \ref{prop3.1}, and H\"older's inequality, we can obtain
		\begin{equation*}
			\begin{aligned}
				\|I_{21}\|_{L^p}&=\left\|A\int_{0}^{\frac{t}{2}}\nabla_{x',y'}\mathbb{G}(t-s)\circledast(uV)(s)ds\right\|_{L^p}\\
				&	\leq A\int_{0}^{\frac{t}{2}}\left\|\nabla_{x',y'}\mathbb{G}(t-s)\circledast(uV)(s)\right\|_{L^p}ds\\
				&	\leq CA\int_{0}^{\frac{t}{2}}\interleave\nabla_{x',y'}\mathbb{G}(t-s)\interleave_{L^{p}}\|\left( uV\right) (s)\|_{L^1}ds\\
				&	\leq CA^{\frac{1}{p}}\int_{0}^{\frac{t}{2}}(t-s)^{-\frac{1}{2}-2(1-\frac{1}{p})}\|\left( uV\right) (s)\|_{L^1}ds\\
				&	\leq CA^{\frac{1}{p}}t^{-\frac{1}{2}-2(1-\frac{1}{p})}\int_{0}^{\frac{t}{2}}\| V(s)\|_{L^{\frac{4}{3}}}\|u(s)\|_{L^4}ds.
			\end{aligned}
		\end{equation*}
		By using Hardy-Littlewood-Sobolev inequality and \eqref{4.2}, we have
		\begin{equation}\label{4.9}
			\begin{aligned} 
				\|I_{21}\|_{L^p}&\leq CA^{\frac{1}{p}}t^{-\frac{1}{2}-2(1-\frac{1}{p})}\int_{0}^{\frac{t}{2}}\| V(s)\|_{L^{\frac{4}{3}}}^2ds\\
				&\leq CA^{\frac{1}{p}}t^{-\frac{1}{2}-2(1-\frac{1}{p})}\ln(1+t)\delta^2\epsilon^2\\
				&\leq CA^{\frac{3}{4}}t^{-2(1-\frac{1}{p})}\delta^2\epsilon^2\\
				&\leq\frac{\delta}{4}t^{-2(1-\frac{1}{p})}\epsilon,
			\end{aligned}
		\end{equation}
		where $p\geq \frac{4}{3}$ was employed.
		On the other hand, 
		\begin{equation}\label{4.10}
			\begin{aligned}
				\|I_{21}\|_{L^p}&	\leq CA\int_{0}^{\frac{t}{2}}\interleave\nabla_{x',y'}\mathbb{G}(t-s)\interleave_{L^{\frac{4}{3}}}\| V(s)\|_{L^p}\|u(s)\|_{L^4}ds\\
				&	\leq CA^{\frac{3}{4}}\int_{0}^{\frac{t}{2}}(t-s)^{-\frac{1}{2}-2(1-\frac{3}{4})}\| V(s)\|_{L^p}\|u(s)\|_{L^{\frac{4}{3}}}ds.\\
				&	\leq CA^{\frac{3}{4}}t^{-1}\int_{0}^{\frac{t}{2}}\| V(s)\|_{L^{\frac{4}{3}}}\|V(s)\|_{L^p}ds\\
				&	\leq CA^{\frac{3}{4}}t^{-1}t\delta^2\epsilon^2\\
				&\leq\frac{\delta}{4}\epsilon.
			\end{aligned}
		\end{equation}
		Combining (\ref{4.9}) and (\ref{4.10}), we have
		\begin{equation}\label{4.11}
			\| I_{21}\|_{L^p}\leq \frac{\delta}{4}(1+t)^{-2(1-\frac{1}{p})}\epsilon.
		\end{equation}
		For $I_{22}$, by using Lemma \ref{lem2.1}, Lemma \ref{lem2.3}, Propsition \ref{prop3.1}, H\"older's inequality and Hardy-Littlewood-Sobolev inequality, we have
		\begin{equation}\label{4.102}
			\begin{aligned}
				\|I_{22}\|_{L^p}&\leq CA\int_{\frac{t}{2}}^{t}\interleave\nabla_{x',y'}\mathbb{G}(t-s)\interleave_{L^{\frac{10}{9}}}\|V(s)\|_{L^p}\|u(s)\|_{L^{10}}ds\\
				&	\leq CA^{\frac{9}{10}}\int_{\frac{t}{2}}^{t}(t-s)^{-\frac{1}{2}-2(1-\frac{9}{10})}\|V(s)\|_{L^p}\|V(s)\|_{L^{\frac{5}{3}}}ds\\
				&	\leq CA^{\frac{9}{10}}(1+t)^{-2(1-\frac{1}{p})}\delta^2\epsilon^2\int_{\frac{t}{2}}^{t}(t-s)^{-\frac{7}{10}}ds\\
				&	\leq CA^{\frac{3}{4}}(1+t)^{-2(1-\frac{1}{p})}\delta^2\epsilon^2\\
				&\leq \frac{\delta}{4}(1+t)^{-2(1-\frac{1}{p})}\epsilon,
			\end{aligned}
		\end{equation}
		where we have used $t\leq2 A^{-\frac{1}{2}}$ .
		Combining (\ref{4.8}), (\ref{4.11}) and (\ref{4.102}), we can obtain that
		$$\|V(x,y,t)\|_{L^p}\leq  \frac{\delta}{2}(1+t)^{-2(1-\frac{1}{p})}\epsilon.$$
		We complete the proof of the proposition.
	\end{proof}
	\begin{proposition}\label{pro4.2}
		Assuming that $t>2 A^{-\frac{1}{2}}$ and (\ref{4.1}) , (\ref{4.2}) hold, then we can infer$$||(m(t, \cdot,\cdot), \omega(t, \cdot,\cdot))||_{L^p}\leq \frac{\delta}{2}(1+t)^{-2(1-\frac{1}{p})}\epsilon,\ \ \ \\ \forall p\geq \frac{4}{3}.$$
	\end{proposition}
	\begin{proof}
		By Duhamel's principle and  intergration by parts, we have
		\begin{equation*}
			\begin{aligned}
				&V(x,y,t)=\int_{\mathbb{R}^2}\mathbb{G}(x-x',y,t;y')V_{0}(x',y')dx'dy'\\
				&+\int_{0}^{t}\int_{\mathbb{R}^2}-\mathbb{G}(x-x',y,t-s;y')\nabla\cdot\left( uV\right) (x',y',s)dx'dy'ds\\
				&=\int_{\mathbb{R}^2}\mathbb{G}(x-x',y,t;y')V_{0}(x',y')dx'dy'\\
				&+\int_{0}^{t}\int_{\mathbb{R}^2}\nabla_{x',y'}\mathbb{G}(x-x',y,t-s;y')\left( uV\right) (x',y',s)dx'dy'ds\\
				&=\mathbb{G}(t)\circledast V_0+\int_{0}^{t}\nabla_{x',y'}\mathbb{G}(t-s)\circledast(uV)(s)ds\\
				&=J_1+J_{2}.
			\end{aligned}
		\end{equation*}	
		For $J_1$, similar to the estimate of $I_1$ in Proposition \ref{pro4.1}, we have
		\begin{equation}\label{4.13}
			\| J_1\|_{L^p}\leq \frac{\delta}{2}(1+t)^{-2(1-\frac{1}{p})}\epsilon.
		\end{equation}	
		For $J_2$, we have
		\begin{equation*}\label{eq:31}
			\begin{aligned}
				J_2
				&=A\int_{t-A^{-\frac{1}{2}}}^{t}\nabla_{x',y'}\mathbb{G}(t-s)\circledast(uV)(s)ds
				+A\int_{t/2}^{t-A^{-\frac{1}{2}}}\nabla_{x',y'}\mathbb{G}(t-s)\circledast(uV)(s)ds\\
				&+A\int_{0}^{\frac{t}{2}}\nabla_{x',y'}\mathbb{G}(t-s)\circledast(uV)(s)ds\\
				&=J_{21}+J_{22}+J_{23}.
			\end{aligned}
		\end{equation*}
		For $J_{23}$, similar to the estimate of $I_{21}$ in Proposition \ref{pro4.1}, we have
		\begin{equation}\label{4.14}
			\| J_{23}\|_{L^p}\leq \frac{\delta}{4}(1+t)^{-2(1-\frac{1}{p})}\epsilon.
		\end{equation}
		For $J_{21}$, by using Lemma \ref{lem2.1}, Lemma \ref{lem2.3}, Propsition \ref{prop3.1}, H\"older's inequality and Hardy-Littlewood-Sobolev inequality, we have
		\begin{equation*}
			\begin{aligned}
				\|J_{21}\|_{L^p}&\leq CA\int_{t-A^{-\frac{1}{2}}}^{t}\interleave\nabla_{x',y'}\mathbb{G}(t-s)\interleave_{L^{p'}}\|V(s)\|_{L^p}\|u(s)\|_{L^{\frac{1}{1-\frac{1}{p'}}}}ds\\
				&	\leq CA^{\frac{1}{p'}}\int_{t-A^{-\frac{1}{2}}}^{t}(t-s)^{-\frac{1}{2}-2(1-\frac{1}{p'})}\|V(s)\|_{L^p}\|V(s)\|_{L^{\frac{1}{1-\frac{1}{p'}+\frac{1}{2}}}}ds.
			\end{aligned}
		\end{equation*}
		To ensure the integrability of time at $s=t,$  and considering the requirement $1< p'<4/3$, we can deduce that
		\begin{equation}\label{4.15}
			\begin{aligned}
				\|J_{21}\|_{L^p}&	\leq CA^{\frac{1}{p'}-\frac{1}{p'}+\frac{3}{4}}(1+t)^{-2(1-\frac{1}{p})}\delta^2\epsilon^2\\
				&\leq  \frac{\delta}{8}(1+t)^{-2(1-\frac{1}{p})}\epsilon.
			\end{aligned}
		\end{equation}
		For $J_{22}$, by using Lemma \ref{lem2.1}, Lemma \ref{lem2.3}, Propsition \ref{prop3.1}, H\"older's inequality and Hardy-Littlewood-Sobolev inequality, we have
		\begin{equation*}
			\begin{aligned}
				\|J_{22}\|_{L^p}&\leq CA\int_{\frac{t}{2}}^{t-A^{-\frac{1}{2}}}\interleave\nabla_{x',y'}\mathbb{G}(t-s)\interleave_{L^{p_1}}\|V(s)\|_{L^p}\|u(s)\|_{L^{\frac{1}{1-\frac{1}{p_1}}}}ds\\
				&	\leq CA^{\frac{1}{p_1}}\int_{\frac{t}{2}}^{t-A^{-\frac{1}{2}}}(t-s)^{-\frac{1}{2}-2(1-\frac{1}{p_1})}\|V(s)\|_{L^p}\|V(s)\|_{L^{\frac{1}{1-\frac{1}{p_1}+\frac{1}{2}}}}ds.
			\end{aligned}
		\end{equation*}
		To utilize Hardy-Littlewood-Sobolev inequality and complete our proof, we consider $4/3< p_1<2$, then we have
		\begin{equation}\label{4.16}
			\begin{aligned}
				\|J_{22}\|_{L^{p_1}}&	\leq CA^{\frac{1}{p_1}-\frac{1}{p_1}+\frac{3}{4}}(1+t)^{-2(1-\frac{1}{p})}\delta^2\epsilon^2\\
				&\leq  \frac{\delta}{8}(1+t)^{-2(1-\frac{1}{p})}\epsilon.
			\end{aligned}
		\end{equation}
		Combining (\ref{4.13})-(\ref{4.16}) , we can obtain that
		$$\|V(x,y,t)\|_{L^p}\leq  \frac{\delta}{2}(1+t)^{-2(1-\frac{1}{p})}\epsilon.$$
		We complete the proof of the proposition.
	\end{proof}
	
	\nocite{*}
	\bibliography{ref.bib}

@ARTICLE{Fortunato2010,
  author  = {Fortunato, S.},
  title   = {Community detection in graphs},
  journal = {Phys. Rep.-Rev. Sec. Phys. Lett.}, 
  volume  = {486},
  year    = {2010},
  pages   = {75-174}
}

@ARTICLE{NewmanGirvan2004,
  author  = {Newman, M. E. J. and Girvan, M.},
  title   = {Finding and evaluating community structure in networks},
  journal = {Phys. Rev. E.}, 
  volume  = {69},
  year    = {2004},
  pages   = {026113}
}

@ARTICLE{Vehlowetal2013,
  author  = {Vehlow, C. and Reinhardt, T. and Weiskopf, D.},
  title   = {Visualizing Fuzzy Overlapping Communities in Networks},
  journal = {IEEE Trans. Vis. Comput. Graph.}, 
  volume  = {19},
  year    = {2013},
  pages   = {2486-2495}
}

@ARTICLE{Raghavanetal2007,
  author  = {Raghavan, U. and Albert, R. and Kumara, S.},
  title   = {Near linear time algorithm to detect community structures in large-scale networks},
  journal = {Phys. Rev E.}, 
  volume  = {76},
  year    = {2007},
  pages   = {036106}
}

@ARTICLE{SubeljBajec2011a,
  author  = {\v{S}ubelj, L. and Bajec, M.},
  title   = {Robust network community detection using balanced propagation},
  journal = {Eur. Phys. J. B.}, 
  volume  = {81},
  year    = {2011},
  pages   = {353-362}
}

@ARTICLE{Louetal2013,
  author  = {Lou, H. and Li, S. and Zhao, Y.},
  title   = {Detecting community structure using label propagation with weighted coherent neighborhood propinquity},
  journal = {Physica A.}, 
  volume  = {392},
  year    = {2013},
  pages   = {3095-3105}
}

@ARTICLE{Clausetetal2004,
  author  = {Clauset, A. and Newman, M. E. J. and Moore, C.},
  title   = {Finding community structure in very large networks},
  journal = {Phys. Rev. E.}, 
  volume  = {70},
  year    = {2004},
  pages   = {066111}
}

@ARTICLE{Blondeletal2008,
  author  = {Blondel, V. D. and Guillaume, J. L. and Lambiotte, R. and Lefebvre, E.},
  title   = {Fast unfolding of communities in large networks},
  journal = {J. Stat. Mech.-Theory Exp.}, 
  volume  = {2008},
  year    = {2008},
  pages   = {P10008}
}

@ARTICLE{SobolevskyCampari2014,
  author  = {Sobolevsky, S. and Campari, R.},
  title   = {General optimization technique for high-quality community detection in complex networks},
  journal = {Phys. Rev. E.}, 
  volume  = {90},
  year    = {2014},
  pages   = {012811}
}

@ARTICLE{FortunatoBarthelemy2007,
  author  = {Fortunato, S. and Barthelemy, M.},
  title   = {Resolution limit in community detection},
  journal = {Proc. Natl. Acad. Sci. U. S. A.}, 
  volume  = {104},
  year    = {2007},
  pages   = {36-41}
}

@ARTICLE{SubeljBajec2011b,
  author  = {\v{S}ubelj, L. and Bajec, M.},
  title   = {Unfolding communities in large complex networks: Combining defensive and offensive label propagation for core extraction},
  journal = {Phys. Rev. E.}, 
  volume  = {83},
  year    = {2011},
  pages   = {036103}
}

@ARTICLE{WangLi2013,
  author  = {Wang, X. and Li, J.},
  title   = {Detecting communities by the core-vertex and intimate degree in complex networks},
  journal = {Physica A.}, 
  volume  = {392},
  year    = {2013},
  pages   = {2555-2563}
}

@ARTICLE{Lietal2013,
  author  = {Li, J. and Wang, X. and Eustace, J.},
  title   = {Detecting overlapping communities by seed community in weighted complex networks},
  journal = {Physica A.}, 
  volume  = {392},
  year    = {2013},
  pages   = {6125-6134}
}

@ARTICLE{Fabioetal2013,
  author  = {Fabio, D. R. and Fabio, D. and Carlo, P.},
  title   = {Profiling core-periphery network structure by random walkers},
  journal = {Sci. Rep.}, 
  volume  = {3},
  year    = {2013},
  pages   = {1467}
}

@ARTICLE{Chenetal2013,
  author  = {Chen, Q. and Wu, T. T. and Fang, M.},
  title   = {Detecting local community structure in complex networks based on local degree central nodes},
  journal = {Physica A.}, 
  volume  = {392},
  year    = {2013},
  pages   = {529-537}
}

@ARTICLE{Zhangetal2007,
  author  = {Zhang, S. and Wang, R. and Zhang, X.},
  title   = {Identification of overlapping community structure in complex networks using fuzzy c-means clustering},
  journal = {Physica A.}, 
  volume  = {374},
  year    = {2007},
  pages   = {483-490}
}

@ARTICLE{Nepuszetal2008,
  author  = {Nepusz, T. and Petr\'oczi, A. and N\'egyessy, L. and Bazs\'o, F.},
  title   = {Fuzzy communities and the concept of bridgeness in complex networks},
  journal = {Phys. Rev. E.}, 
  volume  = {77},
  year    = {2008},
  pages   = {016107}
}

@ARTICLE{FabricioLiang2013,
  author  = {Fabricio, B. and Liang, Z.},
  title   = {Fuzzy community structure detection by particle competition and cooperation},
  journal = {Soft Comput.}, 
  volume  = {17},
  year    = {2013},
  pages   = {659-673}
}

@ARTICLE{Sunetal2011,
  author  = {Sun, P. and Gao, L. and Han, S.},
  title   = {Identification of overlapping and non-overlapping community structure by fuzzy clustering in complex networks},
  journal = {Inf. Sci.}, 
  volume  = {181},
  year    = {2011},
  pages   = {1060-1071}
}

@ARTICLE{Wangetal2013,
  author  = {Wang, W. and Liu, D. and Liu, X. and Pan, L.},
  title   = {Fuzzy overlapping community detection based on local random walk and multidimensional scaling},
  journal = {Physica A.}, 
  volume  = {392},
  year    = {2013},
  pages   = {6578-6586}
}

@ARTICLE{Psorakisetal2011,
  author  = {Psorakis, I. and Roberts, S. and Ebden, M. and Sheldon, B.},
  title   = {Overlapping community detection using Bayesian non-negative matrix factorization},
  journal = {Phys. Rev. E.}, 
  volume  = {83},
  year    = {2011},
  pages   = {066114}
}

@CONFERENCE{ZhangYeung2012,
  author  = {Zhang, Y. and Yeung, D.},
  title   = {Overlapping Community Detection via Bounded Nonnegative Matrix Tri-Factorization},
  booktitle = {In Proc. ACM SIGKDD Conf.}, 
  year    = {2012},
  pages   = {606-614}
}

@ARTICLE{Liu2010,
  author  = {Liu, J.},
  title   = {Fuzzy modularity and fuzzy community structure in networks},
  journal = {Eur. Phys. J. B.}, 
  volume  = {77},
  year    = {2010},
  pages   = {547-557}
}

@ARTICLE{Havensetal2013,
  author  = {Havens, T. C. and Bezdek, J. C. and Leckie, C. and Ramamohanarao, K. and Palaniswami, M.},
  title   = {A Soft Modularity Function For Detecting Fuzzy Communities in Social Networks},
  journal = {IEEE Trans. Fuzzy Syst.}, 
  volume  = {21},
  year    = {2013},
  pages   = {1170-1175}
}

@misc{Newman2013,
  author = {Newman, M. E. J.},
  title  = {Network data},
  howpublished = "\url{http://www-personal.umich.edu/~mejn/netdata/}",
  year = {2013}
}

@ARTICLE{SubeljBajec2012,
  author  = {\v{S}ubelj, L. and Bajec, M.},
  title   = {Ubiquitousness of link-density and link-pattern communities in real-world networks},
  journal = {Eur. Phys. J. B.}, 
  volume  = {85},
  year    = {2012},
  pages   = {1-11}
}

@ARTICLE{Lancichinettietal2008,
  author  = {Lancichinetti, A. and Fortunato, S. and Radicchi, F.},
  title   = {Benchmark graphs for testing community detection algorithms},
  journal = {Phys. Rev. E.}, 
  volume  = {78},
  year    = {2008},
  pages   = {046110}
}

@ARTICLE{Liuetal2014,
  author  = {Liu, W. and Pellegrini, M. and Wang, X.},
  title   = {Detecting Communities Based on Network Topology},
  journal = {Sci. Rep.}, 
  volume  = {4},
  year    = {2014},
  pages   = {5739}
}

@ARTICLE{Danonetal2005,
  author  = {Danon, L. and Diaz-Guilera, A. and Duch, J. and Arenas, A.},
  title   = {Comparing community structure identification},
  journal = {J. Stat. Mech.-Theory Exp.}, 
  volume  = {},
  year    = {2005},
  pages   = {P09008}
}

@ARTICLE{Gregory2011,
  author  = {Gregory, S.},
  title   = {Fuzzy overlapping communities in networks},
  journal = {J. Stat. Mech.-Theory Exp.}, 
  volume  = {},
  year    = {2011},
  pages   = {P02017}
}

@ARTICLE{LancichinettiFortunato2009,
  author  = {Lancichinetti, A. and Fortunato, S.},
  title   = {Benchmarks for testing community detection algorithms on directed and weighted graphs with overlapping communities},
  journal = {Phys. Rev. E.}, 
  volume  = {80},
  year    = {2009},
  pages   = {016118}
}

@CONFERENCE{HullermeierRifqi2009,
  author  = {Hullermeier, E. and Rifqi, M.},
  title   = {A Fuzzy Variant of the Rand Index for Comparing Clustering Structures},
  booktitle = {in Proc. IFSA/EUSFLAT Conf.}, 
  year    = {2009},
  pages   = {1294-1298}
}

@book{Bollob.1993,
  author    = {B. Bollob{\'a}s and W. Fulton and F. Kirwan and P. Sarnak and B. Simon and B. Totaro},
  title     = {Fourier Integrals in Classical Analysis},
  edition   = {Second},
  series    = {Cambridge Tracts in Mathematics},
  publisher = {Cambridge University Press},
  year      = {1993}
}

@book{8_pa,
  author    = {Terence Tao},
  title     = {Nonlinear Dispersive Equations: Local and Global Analysis},
  series    = {CBMS Regional Conference Series in Mathematics},
  number    = {106},
  publisher = {American Mathematical Society},
  address   = {Providence, RI},
  year      = {2006}
}

@book{Stein.1970,
  author    = {E. M. Stein},
  title     = {Singular Integrals and Differentiability Properties of Functions},
  series    = {Princeton Mathematical Series},
  number    = {30},
  publisher = {Princeton University Press},
  address   = {Princeton, NJ},
  year      = {1970}
}

@article{7,
  author  = {A. C. Eringen},
  title   = {Theory of micropolar fluids},
  journal = {Journal of Mathematics and Mechanics},
  volume  = {16},
  number  = {1},
  pages   = {1--18},
  year    = {1966}
}

@article{8,
  author  = {A. C. Eringen},
  title   = {Micropolar fluids with stretch},
  journal = {International Journal of Engineering Science},
  volume  = {7},
  number  = {1},
  pages   = {115--127},
  year    = {1969}
}

@article{9,
  author  = {B. Dong and Z. Chen},
  title   = {Regularity criteria of weak solutions to the three-dimensional micropolar flows},
  journal = {Journal of Mathematical Physics},
  volume  = {50},
  pages   = {103525},
  year    = {2009}
}

@article{10,
  author  = {B. Dong and Z. Chen},
  title   = {Asymptotic profiles of solutions to the 2D viscous incompressible micropolar fluid flows},
  journal = {Discrete and Continuous Dynamical Systems},
  volume  = {23},
  pages   = {765--784},
  year    = {2009}
}

@article{11,
  author  = {B. Dong and J. Li and J. Wu},
  title   = {Global well-posedness and large-time decay for the 2D micropolar equations},
  journal = {Journal of Differential Equations},
  volume  = {262},
  pages   = {3488--3523},
  year    = {2017}
}

@article{12,
  author  = {B. Dong and Z. Zhang},
  title   = {Global regularity of the 2D micropolar fluid flows with zero angular viscosity},
  journal = {Journal of Differential Equations},
  volume  = {249},
  pages   = {200--213},
  year    = {2010}
}

@article{13,
  author  = {X. Jin and Q. Jiu},
  title   = {Stability for the 2D micropolar equations with partial dissipation near Couette flow},
  journal = {Communications in Mathematical Sciences},
  volume  = {22},
  number  = {6},
  pages   = {1529--1548},
  year    = {2024}
}

@article{14,
  author  = {Y. Wang and L. Li},
  title   = {Linear and nonlinear enhanced dissipation for the 2D micropolar equations near Couette},
  journal = {Discrete and Continuous Dynamical Systems Series B},
  volume  = {32},
  pages   = {103--118},
  year    = {2026}
}

@article{16,
  author  = {G. Wang and L. Wang},
  title   = {On asymptotic stability of Couette flow for 2-D Boussinesq system in whole space via Green's function},
  journal = {Discrete and Continuous Dynamical Systems Series B},
  volume  = {30},
  number  = {8},
  pages   = {3014--3041},
  year    = {2025}
}

@article{Vinograd1957,
  author  = {R. E. Vinograd},
  title   = {On the central characteristic index of a system of differential equations},
  journal = {Matematicheskii Sbornik},
  volume  = {42},
  number  = {2},
  pages   = {207--222},
  year    = {1957},
  note    = {(In Russian)}
}

@article{Josic2002,
  author  = {K. Josi{\'c} and R. Rosenbaum},
  title   = {Unstable equilibria with stable characteristic roots},
  journal = {SIAM Review},
  volume  = {44},
  number  = {4},
  pages   = {681--688},
  year    = {2002}
}

@book{amann1990ode,
  author    = {Herbert Amann},
  title     = {Ordinary Differential Equations},
  series    = {De Gruyter Studies in Mathematics},
  number    = {13},
  publisher = {Walter de Gruyter \& Co.},
  address   = {Berlin},
  year      = {1990},
  pages     = {xiv+458},
  isbn      = {3-11-011515-8}
}

@book{Schmid.2001,
  author    = {P. J. Schmid and D. S. Henningson},
  title     = {Stability and Transition in Shear Flows},
  series    = {Applied Mathematical Sciences},
  number    = {142},
  publisher = {Springer-Verlag},
  address   = {New York},
  year      = {2001}
}

@article{Trefethen.1993,
  author  = {L. N. Trefethen and A. E. Trefethen and S. C. Reddy and T. A. Driscoll},
  title   = {Hydrodynamic stability without eigenvalues},
  journal = {Science},
  volume  = {261},
  number  = {5121},
  pages   = {578--584},
  year    = {1993}
}

@article{Kelvin.1887,
  author  = {Lord Kelvin},
  title   = {Stability of fluid motion---rectilinear motion of viscous fluid between two parallel plates},
  journal = {Philosophical Magazine},
  volume  = {24},
  pages   = {188--196},
  year    = {1887}
}

@article{Bedrossian.2019,
  author  = {J. Bedrossian and P. Germain and N. Masmoudi},
  title   = {Stability of the Couette flow at high Reynolds number in 2D and 3D},
  journal = {Bulletin of the American Mathematical Society},
  volume  = {56},
  number  = {3},
  pages   = {373--414},
  year    = {2019}
}

@article{Duguet.2010,
  author  = {Y. Duguet and L. Brandt and B. Larsson},
  title   = {Towards minimal perturbations in transitional plane Couette flow},
  journal = {Physical Review E},
  volume  = {82},
  number  = {2},
  pages   = {026316},
  year    = {2010}
}

@incollection{Lundbladh.1994,
  author    = {A. Lundbladh and D. Henningson and S. Reddy},
  title     = {Threshold Amplitudes for Transition in Channel Flows},
  booktitle = {Transition},
  publisher = {Springer},
  address   = {New York},
  year      = {1994},
  pages     = {309--318}
}

@article{S.Orszag.1980,
  author  = {S. Orszag and L. Kells},
  title   = {Transition to turbulence in plane Poiseuille and plane Couette flow},
  journal = {Journal of Fluid Mechanics},
  volume  = {96},
  pages   = {159--205},
  year    = {1980}
}

@article{Reddy.1998,
  author  = {S. Reddy and P. Schmid and J. Baggett and D. Henningson},
  title   = {On stability of streamwise streaks and transition thresholds in plane channel flows},
  journal = {Journal of Fluid Mechanics},
  volume  = {365},
  pages   = {269--303},
  year    = {1998}
}

@book{Yaglom.2012,
  author    = {A. Yaglom},
  title     = {Hydrodynamic Instability and Transition to Turbulence},
  series    = {Fluid Mechanics and Its Applications},
  number    = {100},
  publisher = {Springer},
  address   = {New York},
  year      = {2012}
}

@article{Bedrossian.2020,
  author  = {J. Bedrossian and P. Germain and N. Masmoudi},
  title   = {Dynamics near the subcritical transition of the 3D Couette flow I: Below threshold case},
  journal = {Memoirs of the American Mathematical Society},
  volume  = {266},
  number  = {1294},
  year    = {2020}
}

@article{Bedrossian.2017,
  author  = {J. Bedrossian and P. Germain and N. Masmoudi},
  title   = {On the stability threshold for the 3D Couette flow in Sobolev regularity},
  journal = {Annals of Mathematics},
  volume  = {185},
  number  = {2},
  pages   = {541--608},
  year    = {2017}
}

@article{Bedrossian.2016,
  author  = {J. Bedrossian and N. Masmoudi and V. Vicol},
  title   = {Enhanced dissipation and inviscid damping in the inviscid limit of the Navier-Stokes equations near the two dimensional Couette flow},
  journal = {Archive for Rational Mechanics and Analysis},
  volume  = {219},
  number  = {3},
  pages   = {1087--1159},
  year    = {2016}
}

@article{Bedrossian.2018,
  author  = {J. Bedrossian and F. Wang and V. Vicol},
  title   = {The Sobolev stability threshold for 2D shear flows near Couette},
  journal = {Journal of Nonlinear Science},
  volume  = {28},
  number  = {6},
  pages   = {2051--2075},
  year    = {2018}
}

@article{1key,
  author  = {R. Arbon and J. Bedrossian},
  title   = {Quantitative Hydrodynamic Stability for Couette Flow on Unbounded Domains with Navier Boundary Conditions},
  journal = {Communications in Mathematical Physics},
  volume  = {406},
  number  = {6},
  pages   = {129},
  year    = {2025}
}

@article{2key,
  author  = {G. Wang and W. Wang},
  title   = {Transition threshold for the 2-D Couette flow in whole space via Green's function},
  journal = {Journal of Mathematical Analysis and Applications},
  volume  = {550},
  number  = {1},
  pages   = {129585},
  year    = {2025}
}

@article{03,
  author  = {X. Luo},
  title   = {The Sobolev stability threshold of 2D hyperviscosity equations for shear flows near Couette flow},
  journal = {Mathematical Methods in the Applied Sciences},
  volume  = {43},
  pages   = {6300--6323},
  year    = {2020}
}

@article{O.Reynolds.1883,
  author  = {O. Reynolds},
  title   = {An experimental investigation of the circumstances which determine whether the motion of water shall be direct or sinuous and of the law of resistance in parallel channels},
  journal = {Proceedings of the Royal Society of London},
  volume  = {174},
  pages   = {935--982},
  year    = {1883}
}

\end{document}